\documentclass[12pt, a4paper]{article}

\usepackage[T1]{fontenc}
\usepackage[latin1]{inputenc}
\usepackage[english]{babel}
\usepackage{amsopn}

\usepackage{cite}		

\usepackage{amsmath, amssymb, amsfonts}
\usepackage{amsthm}
\usepackage{latexsym,color,mathrsfs,fancyhdr,enumerate,nicefrac,stmaryrd} 
\usepackage[all]{xy}
\usepackage{mathabx}
\usepackage{cancel}

\theoremstyle{plain}

\newtheorem{sa}{Theorem}[section]
\newtheorem{Thm}[sa]{Theorem}
\newtheorem{Lem}[sa]{Lemma}
\newtheorem{Prp}[sa]{Proposition}
\newtheorem{Cor}[sa]{Corollary}

\newtheorem{Def}[sa]{Definition}
\newtheorem{DefLem}[sa]{Definition/Lemma}

\newtheorem{Not}[sa]{Notation}
\newtheorem{Rem}[sa]{Remark}

\DeclareMathOperator{\Hom}{Hom}

\DeclareMathOperator{\RHom}{\ensuremath{\text{R}\mathcal{H}\text{om}}}

\renewcommand{\H}{\ensuremath{\mathrm{H}}}
\DeclareMathOperator{\CH}{\ensuremath{CH}}

\DeclareMathOperator{\SK}{\ensuremath{SK}}	
\DeclareMathOperator{\KH}{\ensuremath{KH}}	
	
\DeclareMathOperator{\K}{\ensuremath{K}}		
\newcommand{\G}{\ensuremath{\Gamma}}				
\DeclareMathOperator{\Gal}{Gal}	
\DeclareMathOperator{\Br}{\ensuremath{Br}}	
\DeclareMathOperator{\Alb}{\ensuremath{Alb}}

\DeclareMathOperator{\trdeg}{trdeg}	
			
\DeclareMathOperator{\coker}{coker}	
\DeclareMathOperator{\im}{im}

\DeclareMathOperator{\tr}{tr}

\DeclareMathOperator*{\prlim}{\lim\limits_{\substack{\longleftarrow}}}	
	
\DeclareMathOperator*{\inlim}{\lim\limits_{\substack{\longrightarrow}}}

\newcommand{\md}{\mathrm{mod}}	
\newcommand{\kt}{\equiv}				
\newcommand{\ab}{\mathrm{ab}}		
\newcommand{\di}{\mathrm{div}}
\newcommand{\sep}{\mathrm{sep}}	
\newcommand{\alg}{\mathrm{alg}}	
\newcommand{\Di}{\mathrm{Div}}
\newcommand{\tors}{\mathrm{tors}}

\newcommand{\geo}{\mathrm{geo}}	

\newcommand{\rd}{\mathrm{red}}	
	
\newcommand{\gp}{\mathrm{gp}}

\newcommand{\et}{\text{ét}}

\newcommand{\sft}{\mathrm{sft}}	
\newcommand{\Sch}{\mathrm{Sch}}

\newcommand{\Cat}{\mathcal{C}}

\DeclareMathOperator{\Quot}{Quot}	
\DeclareMathOperator{\cd}{cd}

\DeclareMathOperator{\Spec}{Spec}

\DeclareMathOperator{\ch}{char}		
\newcommand{\Q}{\mathbb{Q}}				
				
\renewcommand{\P}{\mathbb{P}}

\newcommand{\F}{\mathbb{F}}				
\newcommand{\Z}{\mathbb{Z}}				
\newcommand{\Ll}{\mathbb{L}}			
\newcommand{\Lb}{\mathbb{L}\text{-}}

\renewcommand{\k}{\kappa}					
\renewcommand{\O}{\mathcal{O}}		
		
\newcommand{\N}{\mathbb{N}}

\newcommand{\x}{\times}

\newcommand{\X}{\mathfrak{X}}

\newcommand{\pp}{\mathfrak{p}}

\newcommand{\fm}{\mathfrak{m}}

\newcommand{\sm}{\setminus}			
\newcommand{\ins}{\subseteq} 		
\newcommand{\sni}{\supseteq} 		
\newcommand{\inj}{\hookrightarrow}	
\newcommand{\injl}{\hookleftarrow}
\newcommand{\climm}{\hookrightarrow}	

\newcommand{\opimm}{\hookrightarrow}	

\newcommand{\srj}{\twoheadrightarrow}	
\newcommand{\srjl}{\twoheadleftarrow}

\newcommand{\iso}{\cong}		
\newcommand{\ol}{\overline}	
\newcommand{\ds}{\displaystyle}

\begin{document}


\begin{titlepage}
\enlargethispage{7cm}
\setlength{\topmargin}{-4cm}
\title{{\bf Cohomological Hasse principle for schemes over valuation rings of higher dimensional local fields}}
\author{Patrick Forré 
\thanks{The author was partially supported by Deutsche Forschungsgemeinschaft (DFG)} 
\date{} }
\maketitle

\begin{abstract}
K. Kato's conjecture about the cohomological Hasse principle for regular connected schemes $\X$ which are flat and proper over the complete discrete valuation rings $\O_N$ of higher local fields $F_N$ is proven. 
This generalizes the work of M. Kerz, S. Saito and U. Jannsen for finite fields to the case of all higher local fields.
For that purpose a $p$-alteration theorem for the local uniformization of schemes over valuation rings of arbitrary finite rank and a corresponding Bertini theorem is developed extending the results of O. Gabber, J. deJong, L. Illusie, M. Temkin, S. Saito, U. Jannsen to the non-noetherian world.
As an application it is shown that certain motivic cohomology groups of varieties over higher local fields are finite. 
This is one of the rare cases where such a result could be shown for schemes without finite or separably closed residue fields.
Furthermore, it will be derived that the kernels of the reciprocity map 
$\rho^X : \SK_N(X) \to \pi_1^\ab(X)$  
and norm map 
$N_{X|F}: \SK_N(X) \to K_N^M(F_N)$ modulo maximal $p'$-divisible subgroups are finite for regular $X$ which are proper over a higher local field $F_N$ with final residue characteristic $p$.
This generalizes results of S. Bloch, K. Kato, U. Jannsen, S. Saito from varieties over finite and local fields to varieties over higher local fields, both of arbitrary dimensions. 
\noindent
\end{abstract}
\tableofcontents
\thispagestyle{empty}
\end{titlepage}


\pagenumbering{arabic}
\pagestyle{headings}
\setcounter{page}{0}


\setcounter{section}{-1}

\section{Introduction}

Searching for a suitable framework for generalizing the famous \emph{Brauer-Hasse-Noether exact sequence} (cf. \cite{Neu11} Satz III 5.8.)
\[ 0 \longrightarrow \Br(K) \longrightarrow \bigoplus_\pp \Br(K_\pp) \longrightarrow \Q/\Z \longrightarrow 0,\]
which expresses a local-global-principle of an algebraic number field $K$, to rather general excellent schemes $X$, K. Kato found the far reaching cohomological approach (cf. \cite{Kat86}) by constructing the following complexes of Gersten-Bloch-Ogus type $C^{r,b}(X,\Z/n)$ for $n$ invertible on $X$:
\[\begin{array}{ccccc} \bigoplus_{x \in X_0} \H^r(\k(x),\Z/n(b)) &\longleftarrow& \bigoplus_{x \in X_1} \H^{r+1}(\k(x),\Z/n(b+1))
& \longleftarrow& \cdots\\
 \cdots &\longleftarrow &\bigoplus_{x \in X_d} \H^{r+d}(\k(x),\Z/n(b+d)) &\longleftarrow &0, \end{array}\]
 where $\bigoplus_{x \in X_a} \H^{r+a}(\k(x),\Z/n(b+a))$ is put in degree $a$. K. Kato conjectured that these complexes are exact, or equivalently, that their homology groups $\KH_a^{r,b}(X,\Z/n)$ vanish for certain values of $a,b,r,n$  (usually $r=\sup_{x \in X_0}\cd \k(x)$ and $b=r-1$ for $n$ invertible on $X$, $a>0$).\\
Since then Kato's conjectures have been studied for many schemes $X$ over several arithmetic ground schemes $S$ like the spectra of finite fields, global fields, the ring of integers of global or local fields and for higher dimensional (global function) fields (see \cite{Kat86}, \cite{JS03}, \cite{JS08}, \cite{Jan09}, \cite{Gei10b}, \cite{KS10}, et al.), having great effect on the understanding of 
finiteness results on motivic cohomology, special values of zeta functions and higher dimensional class field theory.\\
As a generalization of finite fields ($0$-local fields) and local fields ($1$-local fields) in this paper we will treat the case of higher local fields ($N$-local fields) and their complete discrete valuation rings as ground schemes. For $N \ge 1$ a $N$-local field is definied to be the quotient field of a complete discrete valuation ring with a $(N-1)$-local field as a residue field.\\

If $X$ is a scheme over a finite field or over the valuation ring of a $1$-local field, then
\[ \bigoplus_{x \in X_0} \H^1(\k(x),\Z/n) \longleftarrow \bigoplus_{x \in X_1} \H^{2}(\k(x),\Z/n(1))
 \longleftarrow \bigoplus_{x \in X_2} \H^{3}(\k(x),\Z/n(2)) \cdots, \]
is considered as Kato complex ($r=1,b=0$) and the following \emph{cohomological Hasse principles } are known (cf. \cite{KS10}):

\begin{Thm}[U. Jannsen, M. Kerz, S. Saito]
\label{hasse-finite}
\begin{enumerate}
\item If $Y$ is a connected regular variety proper over a finite field $F_0$, then
	\[ \KH_a(Y,\Z/n) \iso \left \{   
			\begin{array}{ll}
			 \Z/n, & \text{ for }  a=0, \\
			 0, & \text{otherwise,}
		\end{array} 	\right. \]
	for $n$ not divisible by the characteristic $p=\ch(F_0)$.\\
	As a consequence, for a proper strict normal crossing variety $Y$ over $F_0$ and any $a \in \Z$ we have
	\[ \KH_a(Y,\Z/n) \iso \H_a(\G_Y,\Z/n), \]
	where $\H_a(\G_Y,\Z/n)$ is the homology of the dual complex of $Y$.
\item If $\X$ is a connected regular scheme flat and proper over the valuation ring $\O_1$ of a local field $F_1$, then
 \[ \KH_a(\X, \Z/n) = 0 \]
 for any $a \in \Z$ and $n$ invertible on $\O_1$.
\end{enumerate}
\end{Thm}

Unfortunately, the analogon of the first result in \ref{hasse-finite} does not hold generally for varieties $X$ over $1$-local fields $F_1$ taking the corresponding Kato complex
\[ \bigoplus_{x \in X_0} \H^2(\k(x),\Z/n(1)) \longleftarrow \bigoplus_{x \in X_1} \H^{3}(\k(x),\Z/n(2))
 \longleftarrow \bigoplus_{x \in X_2} \H^{4}(\k(x),\Z/n(3)) \cdots. \]
 Even in the case where $X$ is a smooth and projective variety over a local field $F_1$ which has a regular connected model $\X$ flat and projective over the valuation ring $\O_1$ with quasi-semistable reduction $Y=\X_{s,\rd}$, where $s$ is the special point of $S=\Spec(\O_1)$, the exact sequence
 \[ \KH_{a+1}(\X,\Z/n) \longrightarrow \KH_{a}(X,\Z/n) \longrightarrow \KH_{a}(Y,\Z/n) \longrightarrow \KH_{a}(\X,\Z/n)\]
and both kinds of Hasse principles in \ref{hasse-finite} now imply for any $a \in \Z$ the isomorphisms
 \[ \KH_a(X,\Z/n) \iso \H_a(\G_Y,\Z/n)\stackrel{\text{i.g.}}{\neq}0.   \]
 These group vanish for varieties with a smooth $Y$ and $a > 0$ (e.g. in the good reduction case), but not for general strict normal crossing varieties $Y$ as reduction. So clearly the cohomological Hasse principle is wrong in this general case of varieties over local fields, and we can not expect it to generalize even further to higher local fields.\\
 
But if we instead consider the generalization of the second result in \ref{hasse-finite} in the case, where $\X$ is a scheme over the complete discrete valuation ring $\O_N$ of a $N$-local field, and if we take the Kato complex
\[ \bigoplus_{x \in \X_0} \H^N(\k(x),\Z/n(N-1)) \longleftarrow \bigoplus_{x \in \X_1} \H^{N+1}(\k(x),\Z/n(N))
 \longleftarrow \cdots, \]
then the following result will be proven.

\begin{Thm}[Kato's conjecture about the cohomological Hasse principle]
Let $\X$ be a regular connected scheme of dimension $d \ge 1$ proper and flat over the complete discrete valuation ring $\O_N$ of a $N$-local field $F_N$, $N>0$.\\ 
Then we have
\[ \KH_a(\X,\Z/n)=0\]
for every $a\in \Z$ and $n$ not divisible by the final residue characteristic $p=\ch(F_0)$ of $F_N$.
\end{Thm}

Since U. Jannsen, S. Saito and K. Sato realized (cf. \cite{JSS09}) that the Kato complex for a scheme $f:\X \to S$ arises (up to sign) as the bottom complex of the $E^1$-layer of the niveau spectral sequence (cf. \cite{BO74}):
\[ E^1_{r,q}(\X)= \bigoplus_{x \in \X_{(r)}} \H_{r+q}(x) \quad\Longrightarrow\quad \H_{r+q}(\X),  \]
 attached to the étale homology theory defined by
\[  \H_a(\X,\Z/n):= \H^{N+2-a}_\et(\X,Rf^!\Z/n(N)_{S}), \]
 Kato's conjectures experienced a vast generalization to arbitrary homology theories. 
So, in this paper an axiomatic approach and prove of Kato's conjectures over $N$-local fields with conditions on the homology theories is given. And they will be verified for the étale homology theory.\\

The main idea of the proof of Kato's conjecture is to relate the occuring scheme $\X$ to varieties over a finite field and strongly make use of weight arguments (cf. \cite{Del80}) and the now proven Bloch-Kato conjecture (cf. \cite{Voe10b}, \cite{Voe10a}, \cite{SJ06}, \cite{HW09}) as done in \cite{KS10}. But since during the usage of a resolution of singularities of schemes over discrete valuation rings (cf. \cite{Hir64}, \cite{CJS09}, \cite{CP08}, \cite{CP09}) or of Gabber-de-Jong's $\ell'$-alterations (cf. \cite{Ill09}, \cite{dJ96})
information gets lost in the process, the idea of doing induction by $N$ and $d$ by applying a Bertini theorem (cf. \cite{JS09} §1, \cite{SS10} §4) etc. gets to a halt. 
To address this problem a new $p$-alteration theorem for the local uniformization of schemes $\hat \X$ over the whole rank-$N$-valuation ring $\O$ of the $N$-local field $F$ is developed:

\begin{Thm}[$p$-alteration for a local uniformization of schemes over valuation rings]
Let $\hat \X$ be integral and flat and of finite type over $\hat S=\Spec(\O)$,
$\hat f: \hat \X \to \hat S$, and let $\hat Z$ be a proper closed subset of $\hat \X$.\\
 Then after suitable $p$-alterations the map $\hat f$ Zariski-locally is given by
 a composition of algebra homomorphisms
\[ \O \to \frac{\O[X_1,\dots,X_s]}
{\left(t_1-M_1,\cdots, t_N-M_N \right)} \to C \tag{A},\]
where the first map is the canonical map and the second map is étale and $t_1,\dots, t_N$ is an ordered system of parameters of $\O$
 and $M_1, \dots, M_N$ are pairwise coprime monomials in the $X_i$ such that if $M_j=\prod_{i=1}^s X_i^{e_{ji}}$ then $\gcd(e_{j1},\dots,e_{js},p)=1$ for all $j=1,\dots,N$. Furthermore, (A) can be choosen such that $\hat Z$ in the middle is given by the zero set $V(X_1\cdots X_r)$ for a $r \le s$.
\end{Thm}

The proof uses noetherian approximation, the strongest version of 
Gabber-deJong-Illusie-Temkin $p$-alteration theorem for quasi-excellent schemes (cf. \cite{ILO12} X Thm. 3.5, \cite{Tem15} Thm. 4.3.1), a resolution result for monoids in totally ordered groups (cf. \cite{GR03} Thm. 6.1.31), resolution of toric singularities (cf. \cite{Kat94}, \cite{Niz06}, \cite{ILO12}) and hands-on concrete computations.\\

Furthermore, the following Bertini theorem over valuation rings is needed and shown:

\begin{Thm}[Bertini theorem over valuation rings]
Let $\hat f:\hat\X \to \hat S$ and $\hat Z$ be locally given as in (A).
Let $i: \hat \X \inj \P^{q}_{\hat S}$ be a fixed embedding, $q \in \N$.\\
Then after a multiple embedding of $i$ there exists a $\O$-rational hyperplane 
$H$ intersecting $\hat \X$ and every reduction of every irreducible component of every multiple intersection of $\hat Z$ transversally and 
such that $\hat f_|: \hat \X \cap H \to \hat S$, $\hat Z \cap H$ and $\hat Z \cup (\hat X \cap H)$ still have a description as in (A).
\end{Thm}

For an application of the cohomological Hasse principle and the alteration theorem we consider Bloch's higher Chow groups (cf. \cite{Blo86}, \cite{Lev04}, \cite{Gei04}) $\CH^q(\X,s)$.
A folklore conjecture states that these groups are finitely generated for schemes over arithmetic ground schemes. A positive result in this direction for $S=\Spec(\O_N)$ the spectrum of a complete discrete valuation ring of a $N$-local field, $p$ the final residue characteristic, is the following.

\begin{Thm}[Finite motivic cohomology groups]
\begin{enumerate}
\item  Bloch's higher Chow groups with finite coefficients
\[ \CH^q(\X,2q-r;\Z/n),\]
are finite for $p \nmid n$ and schemes $\X$ which are separated and of finite type over $\O_N$ of pure dimension and if $0 \le r \le q+1$ or $q \ge N+\dim_S\X-1$. 
\item If $\X$ is connected, regular and proper and flat over $S$, then the étale cycle map (cf. \cite{Blo86}, \cite{GL01}, \cite{Lev04}, \cite{Gei04})
\[ \rho_{\X,n}^{r,q}:\; \CH^q(\X,2q-r;\Z/n) \to \H_\et^{r}(\X,\Z/n(q)) \]
is an isomorphism of finite groups for $p \nmid n$ and if $0 \le r \le q$ or $q \ge N+\dim_S\X-1$, and is injective for $r=q+1$.
\end{enumerate}
\end{Thm}

These results generalize the results of S. Saito and M. Kerz (cf. \cite{KS10}, \cite{CTSS83}, \cite{Sai85b}) where their ground schemes always had finite residue fields to the case of a higher local field of any dimension.\\

As a further application of the cohomological Hasse principle we will consider higher dimensional unramified abelian class field theory of regular and proper varieties $X$ over $N$-local fields $F_N$. For this, let $\pi_1^\ab(X)$ be the abelianized étale fundamental group of $X$ classifying the finite abelian étale covers of $X$, and $\rho^X$ be the push-forward map of the sum of K. Kato's local reciprocity maps 
(cf. \cite{Kat7982})
\[ \rho^x: \K^M_N(\k(x)) \longrightarrow \Gal^\ab_{\k(x)} \]
of closed points $x \in X$. 
The reciprocity law (cf. \cite{Sai85}) says that for proper $X$ this map factors through 
\[\SK_N(X):=\coker\left( \bigoplus_{y \in X_1} \K^M_{N+1}(\k(y)) \stackrel{\partial}{\longrightarrow}  \bigoplus_{x \in X_0} \K^M_N(\k(x)) \right),  \]
where $\partial$ is the boundary map coming from Milnor-$K$-theory and $X_a$ is the set of $a$-dimensional points of $X$.\\
One main challenge of the class field theory of $X$ is the understanding of the kernel 
of $\rho^X$.
Since higher dimensional class field theory (modulo $n$) can be deduced from the above niveau spectral sequence of the étale homology theory (cf. \cite{JS03}) the above cohomological Hasse principle now relate the class field theory of $X$ over a $N$-local field $F_N$ with some reductions $Y$ over the corresponding finite field $F_0$, $p=\ch(F_0)$. And the following theorem will be deduced.
\begin{Thm}
The kernel of the reciprocity map $\ker \rho^X$ is the direct sum of a finite group and a group which is $\ell$-divisible for all $\ell \neq p$.\\
Furthermore, the kernel of the norm map modulo maximal $p'$-divisible subgroup
\[ N_{X|F_N}: \SK_N(X)/\SK_N(X)_{p'-\Di} \to K_N^M(F_N)/K_N^M(F_N)_{p'-\Di} \]
is finite as well.
\end{Thm}
This theorem generalizes the results of S. Bloch, K. Kato, S. Saito, U. Jannsen and the author (cf. \cite{Blo81}, \cite{Sai85}, \cite{JS03}, \cite{For11b}, \cite{For11a}) from curves, surfaces, varieties resp., over finite and local fields to varieties over higher local fields, both of arbitrary dimensions.
\\

 \vfill

{\bf Acknowledgement}\\
The author wants to express his sincere gratitude towards Prof. Dr. Takeshi Saito for the great hospitality during the stay at the University of Tokyo, Japan. He also wants to thank Prof. Dr. Uwe Jannsen, Prof. Dr. Shuji Saito and Prof. Dr. Moritz Kerz for their moral and professional support. Furthermore, the author is very thankful for various inspiring disscussions and helpful email correspondences with Prof. Dr. Michael Temkin, Prof. Dr. Luc Illusie, Prof. Dr. Amaury Thuillier and Dr. Martin Ulirsch.\\
The author was supported by Deutsche Forschungsgemeinschaft (DFG) during his stay at the University of Tokyo, Japan.

\section{Geometry over valuation rings}

\begin{Def}
Let $(\Gamma,+,>)$ be a totally ordered abelian group of finite rank $N$.
\begin{enumerate}
\item A subgroup $H \ins \Gamma$ is called \emph{convex} if for all $x,y \in H$ and $z \in \Gamma$ we have:
\[ x < z < y \quad \Longrightarrow \quad z \in H.  \]
\item An \emph{ordered basis} of $\Gamma$ is a tupel $(v_1,\dots,v_N)$ of $N$ elements $v_i \in \Gamma$ with $0<v_1 < \dots < v_N$ such that
for every convex subgroup $H \ins \Gamma$ there is a subset of $\{v_1,\dots,v_N\}$ building a $\Z$-basis for $H$.
\item Let $F$ be a valued field with valuation $v:\;F^\x \to \Gamma$, and corresponding valuation ring $\O$. 
 An \emph{ordered system of parameters} of $F$ with respect to $v$ are elements $t_1,\dots,t_n \in \O$ such that $(v(t_1),\dots,v(t_n))$ is an ordered basis of $v(F^\x)$.
\end{enumerate}
Note that the prime ideals of $\O$ correspond one-to-one to the convex subgroups of $v(F^\x)$ (cf. \cite{GR03}).
\end{Def}

\begin{Def}
\label{pqss-red}
Let $\O$ be a valuation ring with valuation group $\Gamma$ of finite rank $N$ and $p \ge 0$ the residue characteristic of $\O$. 
An integral scheme $\hat \X$ with a flat map $\hat f:\hat\X \to \hat S:=\Spec(\O)$ of finite type is called with \emph{poly-quasi-semi-stable reduction} if 
$\hat f$ Zariski-locally is given by a composition of algebra homomorphisms
\[ A \to B=\frac{A[X_1,\dots,X_s]}
{\left(t_1-M_1,\cdots, t_N-M_N \right)} \to C \tag{1},\]
where the first map is the canonical map and the second map is étale and $t_1,\dots, t_N$ is an ordered system of parameters of the valuation ring $A$
 and $M_1, \dots, M_N$ are pairwise coprime monomials in the $X_i$ such that if $M_j=\prod_{i=1}^s X_i^{e_{ji}}$ then $\gcd(e_{j1},\dots,e_{js},p)=1$ for all $j=1,\dots,N$. \\
A closed and reduced subscheme $\hat Z \ins \hat\X$ of a scheme with poly-quasi-semi-stable reduction is called a \emph{simple normal crossing divisor} if 
 there is a local representation like (1) such that $\hat Z$ is given by $V(X_1\cdots X_r)$ for a $r \le s$.
\end{Def}

\begin{Rem}
\label{pqss-red-rem}
 Let $\hat f:\hat \X \to \hat S=\Spec(\O)$ be flat of finite type with poly-quasi-semi-stable reduction and $\pp \ins \O$ the height one prime ideal and $\hat W =V(\pp)=\Spec(\O/\pp)$.
\begin{enumerate}
\item $\hat f$ is of finite presentation (cf. \cite{RG71} 3.4.7) and generically smooth.
\item $B$ in \ref{pqss-red} can be written as
\[ B = A \otimes_{\Z[\N^N]}\Z[\N^s],\]
where the generators of the monoids correspond to the $X_i$, the $t_j$, resp., and the images of the $t_j$ are the $M_j$. These define monoidally smooth Zariski fs log structures (cf. \cite{Kat94}, \cite{Niz06}) such that the map is log smooth (cf. \cite{Kat94} §8) and (log) integral (cf. \cite{Kat89} 4.1, \cite{ILO12} X 3.6.1 + 3.6.5) between them.
\item $\hat Y=\hat f^{-1}(\hat W)_\rd$ is a simple normal crossing divisor of $\hat\X$, since locally it corresponds to 
$V(X_1\cdots X_r)$, where $X_1,\dots,X_r$ are exactly the variables occuring in $M_1\cdots M_N$.
\item If $\hat Z \ins \hat \X$ is a simple normal crossing divisior which is flat over $\hat S$, then $\hat Z$ is given by $V(X_{r+1}\cdots X_t)$ for a $t \le s$ and $\hat Z \cup \hat Y$ is also a simple normal crossing divisor, given by $V(X_1\cdots X_t)$.
\item Let $\hat Y =\bigcup_i \hat Y_i$ be the union of its irreducible components $\hat Y_i$ of $\hat Y$. An irreducible component $ \hat V$ of a multiple intersection $\hat Y_{k_1} \cap \dots \cap \hat Y_{k_m}$ is then 
 flat of finite type over $\hat W$ with poly-quasi-semi-stable reduction and simple normal crossing divisor $\hat Z \cap \hat V$.\\
Indeed, locally such a $\hat Y_i$ is then given by $V(X_{in},\dots,X_{iN})$ with some variables $X_{ik} \mid M_k$ for $k=n,\dots,N$, where $t_n,\dots,t_N$ are all the parameters lying in $\pp$. Then $\hat V$ is given by $V(X_{k_1n},\dots, X_{k_mN})$ (and $X_{k_ik} \nmid M_1\cdots M_{n-1}$).  So $\hat f \x_{\hat S} \hat W|_{\hat V}$ is then locally given by
\[ A/\pp \to \frac{A/\pp[X_1,\dots,\cancel{X_{k_1n},\dots,X_{k_mN}},\cdots,X_s ]}{\left(t_1-M_1,\dots,t_{n-1}-M_{n-1}\right)} \to C/\pp C,\]
and $\hat Z \cap \hat V$ in this representation is given by $V(X_{r+1}\cdots X_t)$. 
\item Now assume that $\O_\pp$ is a discrete valuation ring, 
 and consider the pullback 
$\tilde \X=\hat \X \x_{\hat S} \tilde S$ of $\hat \X$ to $\tilde S=\Spec(\O_\pp)$.
Then $\tilde \X$ is regular and its closed subscheme $\tilde Z = (\hat Z \x_{\hat S} \tilde S)_\rd$ is a simple normal crossing divisor (both in the classical sense).\\
Indeed, the log structure on $\tilde S$ is Zariski and log regular (\cite{Kat94} §1) and $\tilde f: \tilde \X \to \tilde S$ is log smooth. 
It follows that the log structure on $\tilde \X$ is also log regular (\cite{Kat94} 8.2) and still monoidally smooth. By \cite{Niz06} 5.2, , \cite{ILO12} VIII 3.4.4, resp., $\tilde \X$ is then regular and its closed subscheme $\tilde Z = (\hat Z \x_{\hat S} \tilde S)_\rd$ is a normal crossing divisor, which is even a simple normal crossing divisor since the log structure is Zariski.
\item Note, that the reduced fibres $\hat \X_{m,\rd}$ in general are not simple normal crossing varieties. Even in the case where the valuation ring is "`glued"' by complete discrete valuation rings the multiple intersections of the irreducible components of $\hat \X_{m,\rd}$ are regular, but they might not be of the correct dimension to build a simple normal crossing variety.
\end{enumerate}
\end{Rem}

\begin{Thm}[Bertini theorem over valuation rings]
\label{bertini} 
Let $\O$ be a valuation ring with valuation group of finite rank $N$, $\hat S=\Spec(\O)$ with closed point $s=\pp \in \hat S=\Spec(\O)$, $\hat W =\ol{\{y\}}\ins \hat S$ the codimension one closed subset.
Let $\hat f:\hat\X \to \hat S$ be flat and projective with poly-quasi-semi-stable reduction with simple normal crossing divisor $\hat Z \sni \hat f^{-1}(\hat W)_\rd$. Let $i: \hat \X \inj \P^{q}_{\hat S}$ be a fixed embedding, $q \in \N$.
\begin{enumerate}
\item Let $H$ be a $\O$-rational hyperplane in $\P^q_{\hat S}$ such that $H_{s}$ intersects every irreducible component $V$ of the reduction of every multiple intersection $\left(\hat Z _{i_1} \cap \dots \cap \hat Z_{i_r}\right)_{s}$ transversally in $\hat \X_{s}$. 
Then $\hat Z \cup (H\cap \hat \X)$ is a simple normal crossing divisor of $\hat \X$ and $\hat f|: H\cap\hat\X \to \hat S$ has poly-quasi-semi-stable reduction with $\dim (H\cap\hat \X) = \dim\hat\X-1$ and simple normal crossing divisor $H \cap \hat Z$.
\item If $\O/\pp$ is infinite then there exists a $\O$-rational hyperplane $H$ like in 1.) intersecting all the $V$ transversally and such that $V \cap H$ is irreducible whenever $\dim V \ge 2$, and $H \cap \hat\X$ is integral in case $\dim \hat\X_{y} \ge 2$.
\item If $\O/\pp$ is finite then the same statement as in 2.) is true after replacing $i$ by a multiple embedding of $i$.
\end{enumerate}
\begin{proof}
1.) Similar arguments as in \cite{JS09} Thm. 1 apply. Let $x \in \hat\X$ be a closed point and 
\[ A \to B=\frac{A[X_1,\dots,X_s]}
{\left(t_1-M_1,\cdots, t_N-M_N \right)} \to C \tag{1}\]
be the local description. 
$x$ then corresponds in $\Spec(B)$ to the ideal $\fm=\pp+(X_1,\dots,X_s)$ and $\hat Z$ to $V(X_1\cdots X_r)$. The reduction of the intersection of all irreducible components of $\hat Z$ in $\Spec(B)$ is given by $V=V(I)$, where 
$I=\pp+(X_1,\dots, X_r)$. Then $\dim V= \dim B/I =s-r$.\\
Let the hyperplane $H$ in $\Spec(B)$ be given by $V(h)$.
In case $H$ does not run through $x$ (e.g. when $\dim V =0$) then $h \notin \fm$ and therefore is invertible in a neighbourhood of $x$ in $\Spec(B)$. So replacing $C$ by a localization the description (1) stays the same.\\
For $s-r \ge 1$ the fact that $H$ intersects $V$ transversally translates to 
$0 \neq \ol h \in \fm/(\fm^2+I)$. Then $h \kt a_{s-r+1}X_{s-r+1}+\dots +a_sX_s \;\md\;( \fm^2+I)$ with some $a_i \in A$, where at least one, let say $a_s$, is not vanishing. So we can assume that $a_s \in A^\x$. 
For short let $\tilde B$ be like $B$ as in (1) but without the $X_s$. Then $B=\tilde B[X_{s}]$. The element $d:=\frac{\partial h}{\partial X_{s}}\kt a_s \notin \fm$. 
Then the map $\hat f$ around $x$ localizes to 
\[A \to B':=\tilde B[T] \to \tilde B[X_s]_d=B_d \to C_d ,\qquad T \mapsto h ,\]
where the last two maps are étale. The divisor $\hat Z\cup (H\cap \hat\X)$ is in $\Spec(B')$ then given by $V(X_1 \cdots X_r\cdot T)$.\\
Similarly the map $\hat f|$ around $x$ on $H \cap \X$ is locally given by
\[A \to \tilde B = \tilde B[T]/(T) \to \tilde B[X_s]_d/(h)=B_d/(h) \to C/(h), \]
where the last two maps are étale.  The closed subscheme $H \cap \hat Z$ of $H\cap \X$ is in $\Spec(\tilde B)$ still given by $V(X_1 \cdots X_r)$. \\
2.+3.) The existence of the hypersurface $H_s$ follows from the Bertini theorem over the residue field $\O/\pp$ with
 \cite{Jou83} and \cite{Poo04}. Any $\O$-rational lift $H$ for $H_s$ will by 1.) then do.
\end{proof}
\end{Thm}

\begin{Def} Let $p \ge 0$ be either a prime or $p=0$.
\begin{enumerate} 
\item A morphism $\hat f: \hat \X \to \hat S$ is called \emph{maximally dominating}, if it sends generic points to generic points.
\item A morphism $b: \hat \X' \to \hat \X$ is called \emph{alteration}, if it is quasi-compact, separated, universally closed, maximally dominating, surjective and generically finite.
Carefully note that we explicitely refrain the morphism from being "locally of finite type" (i.e. from being proper,  cf. \cite{ILO12} X 3.3.1)! This is done to include the case in \ref{int-clo-alt}.
\item A morphism $b: \hat \X' \to \hat \X$ is called \emph{$p$-alteration}
if it is an alteration and the induced residue field extensions of the generic points are all of $p$-power degrees (where "$p$-power" is defined to be $1$ in case $p=0$). 
\item A morphism $\hat f': \hat \X' \to \hat S'$ is called \emph{pseudo-projective} if it factors as a local isomorphism followed by a projective map (cf. \cite{ILO12} X §3.1.6). Note that by \cite{ILO12} X 3.1.7 (ii) a quasi-projective map is the same as a separated pseudo-projective map.
\end{enumerate}
\end{Def}

\begin{Lem}
\label{int-clo-alt}
Let $A \ins B$ be integral domains with $B$ integral over $A$ and 
$[\Quot(B):\Quot(A)] < \infty$.
Then the induced map $g: \Spec(B) \to \Spec(A)$ is an integral alteration (but not necessarily locally of finite type).
\begin{proof}
By assumption $g$ is integral, generically finite, maximally dominating and
 affine, thus separated and quasi-compact. 
$g$ is surjective by the going-up theorem (cf. \cite{AM69} Thm. 5.10) and universally closed by the valuation criterion (cf. \cite{EGA2} 7.3.8).
\end{proof}
\end{Lem}

\begin{Thm}[$p$-altered local uniformization of schemes over valuation rings]
\label{alteration}
Let $\O$ be a valuation ring with valuation group $\Gamma$ of finite rank $N \in \N$, $\hat S := \Spec(\O)$ and $p \ge 0$ its residue characteristic. 
Let $\hat f: \hat \X \to \hat S$ be a maximally dominating morphism of finite type and $\hat Z \ins \hat\X$ a nowhere-dense closed subset. 
Then there exist: 
\begin{enumerate}
  \item[(a)] an integral $p$-alteration $a: \hat S' \to \hat S$, where $\hat S'=\Spec(\O')$ with $\O'$ the integral closure of $\O$ in a finite field
	extension of $p$-power degree (and thus an integrally closed semi-local ring),
	\item[(b)] a $p$-alteration $b: \hat \X' \to \hat \X$ which is the compositum of projective maps with a base change of $a$,
 \item[(c)] a pseudo-projective and flat $\hat f': \hat \X' \to \hat S'$ of finite presentation with $\hat f \circ b = a \circ \hat f'$:
\end{enumerate}
\[ \xymatrix{
\hat\X' \ar^-{b}[r] \ar^-{\hat f'}[d] & \hat\X \ar^-{\hat f}[d] \\
\hat S' \ar^-{a}[r] & \hat S,
}\]
such that after localizing at any closed point of $\hat S'$ the scheme $\hat \X'$ has poly-quasi-semi-stable reduction via $\hat f'$ and simple normal crossing divisor
\[\hat Z' =\left(b^{-1}(\hat Z \cup \hat f^{-1}(\hat W))\right)_\rd,\]
 where $\hat W \ins \hat S$ is the codimension-one closed subscheme. 
\end{Thm}

\begin{Rem}
\label{alt-rem}
\begin{enumerate}
\item If $\hat f$ in \ref{alteration} is separated (proper, resp.) then $\hat f'$ can be assumed to be quasi-projective (projective, resp.) by \cite{ILO12} X 3.1.7 (ii).
\item If for some reason in \ref{alteration} the map $a$ is locally of finite type, then $a$ is finite and $b$ is projective.
 Indeed, a map is finite if and only if it is integral and locally of finite type.
\item If $\O$ is henselian in \ref{alteration} then $\O'$ is a valuation ring as well.
\item The alteration theorem \ref{alteration} will be applied to the henselian valuation ring $\O=\O_N^{(N)}$ of a $N$-local field $F_N$ (see \ref{hlf-def}). The map $a$ in \ref{alteration} is then given as the extension of valuation rings $a^\#: \O_N^{(N)} \to {\O'}_N^{(N)}=\O'$ of a finite field extension $F'_N|F_N$, which in general is not (!) finite. 
Furthermore, let $\O_m^{(1)}$ be a complete discrete valuation ring in the construction of $\O$ then $\O_m^{(1)}=\O_{\pp_{m-1}}/\pp_m\O_{\pp_{m-1}}$.
The corresponding fibre along $a$ is then givesn by the algebra:
\[ \O' \otimes_O O_m^{(1)} = \O'_{\pp'_{m-1}}/\pp_m\O'_{\pp'_{m-1}}, \] 
which also in general need not (!) to be a finitely generated $\O_m^{(1)}$-module. But we have the following property which will be exploited in its application:
\[(\O' \otimes_O O_m^{(1)})_\rd = (\O'_{\pp'_{m-1}}/\pp_m\O'_{\pp'_{m-1}})_\rd
= \O'_{\pp'_{m-1}}/\pp_m'\O'_{\pp'_{m-1}} = {\O'}_m^{(1)}\]
is a finitely generated $\O_m^{(1)}$-module
since it is the integral closure of the excellent ring $\O_m^{(1)}$ in the finite field extension $F'_m|F_m$ (see \ref{ext-val-rings}).
So the maps 
\[a_{m,\rd}:\; \left(a^{-1}(\Spec(\O_m^{(1)}))\right)_\rd = \Spec({\O'}_m^{(1)}) \to \Spec(\O_m^{(1)}) \]
are finite (!) for $m=1,\dots,N$. 
And the corresponding maps $b_{m,\rd}$ in \ref{alteration} are then projective, 
$m=1,\dots,N$.
\item In the first version of this paper a $\ell'$-alteration theorem was stated and proved with help of \cite{ILO12} X Thm. 3.5. But with help of the newer \cite{Tem15} Thm. 4.3.1 the theorem \ref{alteration} could also be improved to only use $p$-alterations.
\end{enumerate}
\end{Rem}

\begin{proof}[Proof of \ref{alteration}]
W.l.o.g. $\hat Z$ can be replaced by $\hat Z \cup \hat f^{-1}(\hat W)$, where $\hat W$ is the codimension one closed subscheme of $\hat S$. Furthermore, by separating the irreducible components of $\hat X$ and taking the reduced structure one can assume $\hat \X$ to be integral. $\hat f$ then becomes flat (\cite{Bou72} Ch.VI §3.6 Lem. 1). $\hat f$ is then of finite presentation by \cite{RG71} 3.4.7, and the closed subset $\hat Z$ can be given a scheme structure of finite presentation over $\hat S$.\\
Let $\Lambda:=\im(\Z \to \O)_{(p)}$ (i.e. $\Lambda=\Q$, $\F_p$ or $\Z_{(p)}$). Then $\O$ can be written as a union $\O=\bigcup_i R_i$ of finitely generated $\Lambda$-algebras $R_i \ins \O$, which are noetherian, quasi-excellent and by \cite{Tem15} Thm. 1.2.4 universally $p$-resolvable. 
By noetherian approximation \cite{EGA43} Thm. 8.8.2 (ii) there exists an index $i$, a scheme $\X$ with a nowhere-dense closed subscheme $Z \ins \X$ and a maximally dominating morphism of finte type $f: \X \to S=\Spec(R_i)$ such that $\hat \X=\X\x_S \hat S$, $\hat Z=Z \x_S \hat S$, $\hat f=f \x_S \hat S$.\\
By the $p$-alteration theorem of Gabber-deJong-Illusie-Temkin for noetherian quasi-excellent schemes \cite{Tem15} Thm. 4.3.1
(also cf. \cite{ILO12} X Thm. 3.5)
there are projective $p$-alterations $a: S' \to S$, $b:\X' \to \X$ with regular sources, a pseudo-projective morphism $f': \X' \to S'$ compatible with $f, a, b$ and snc divisors $W' \ins S'$ and $Z' \ins \X'$ such that $Z'=b^{-1}(Z) \cup f'^{-1}(W')$ and the morphism $f': (\X', Z') \to (S',W')$ is log smooth (cf. \cite{Kat89} §3). W.l.o.g we now can assume $\X'$ to be projective over $S'$.
 Locally around a closed point $f'$ is then (cf. \cite{Kat89} 3.5 + 3.6) given by the compositum of ring homomorphisms
\[ A \to B=A\otimes _{\Z[P]}\Z[Q] \to C \tag{2},\]
where the last map is étale, $\varphi: P \to Q$ is a monoid morphisms between fine and saturated monoids whose generators correspond to local equations of the sncd $W'$ and $Z'$.
Furthermore, The corresponding map
$\varphi^\gp: P^\gp \to Q^\gp$ is injective and the torsion of its cokernel is finite of order prime to $p$.\\
We now take the base change with $\x_S \hat S$. 
The projective $p$-alteration $S'\x_S \hat S \to \hat S$ becomes then finite over the valuation ring.
Let $\hat S'$ be the normalisation of an irreducible component of $S'\x_S \hat S$. Then $\hat S'=\Spec(\O')$ with a semi-local ring $\O'$ (cf. \cite{Bou72} VI §8.3 Rem., §8.5 Thm.2). And the map $\hat S' \to \hat S$ is then an integral $p$-alteration by \ref{int-clo-alt} (Note that we here make use of the relaxed definition of $p$-alterations, which might not be locally of finite type). 
Now take another base change with $\x_{S'\x_S \hat S} \hat S'$ and get $\hat \X''=\X' \x_S \hat S'$ over $\hat S'$, which is then locally still given as in (2). \\
In this presentation the ring $A$ can be chosen to be a valuation ring (\cite{Bou72} VI §8.3 Rem.) which is étale and generically finite over $\O'$ with a totally ordered valuation group $\Gamma'$ of rank $N$.\\
Let $a_i$ be the images of the generators of $P$ in $A$. 
Since there are only finitely many $a_i$ by \cite{GR03} Thm. 6.1.31 there is an ordered system of parameters $t_1,\dots,t_N$ of $A$, 
 whose generated monoid $P'$ is free and contains the $a_i$ up to units $\epsilon_i \in A^\x$. 
Because the torsion of $\coker(\varphi^\gp)$ is prime to $p$, there is a étale extension $A \to A'$ such that these units can be transformed into the generators of $Q$ after a basechange to $A'$.
So w.l.o.g. we can assume that the $a_i$ lie in $P'$ and the map $P \to A'$ factors through $P \to P' \ins A'$. 
Now replace $A'$ with $A$ again and let
$Q'=Q\oplus_P P'$ be the monoidal pushout of $Q$ and $P'$ along $P \to P'$. 
Then we have 
\[B=A \otimes_{\Z[P]}\Z[Q]=A \otimes_{\Z[P']}\Z[Q'] \tag{3}\]
and $Q'$ and $P'$ give the log structures on $\hat \X''$ and $\hat S'$.\\
Now the issue that $Q'$ might not be a free monoid is treated.
We apply to $\Spec(C)$ (i.e. to $\Spec(\Z[Q'])$ and then pullback) the monoidal desingularisation functor (cf. \cite{ILO12} VIII 3.4.9), which is composed of a sequence of saturated log blow ups. 
The resulting scheme locally is then of the form $\Spec(\Z[Q''])$ with a free monoid $Q''$ and an injective map $Q' \to Q''$ with $(Q'')^\gp\iso (Q')^\gp$.
Taking the composite of the canonical map to the base change of (3) we get
\[ A \to A \otimes_{\Z[P']}\Z[Q'']=B\otimes_{\Z[Q']}\Z[Q'']   \to C\otimes _{\Z[Q']}\Z[Q'']=:C' \tag{4}\]
where the second map is étale. Explicitely we have
\[B':=A \otimes_{\Z[P']}\Z[Q'']=
\frac{A[X_1,\dots,X_s]}
{\left(t_1-M_1,\cdots, t_N-M_N \right)} \tag{5} \]
with some monomials $M_i$.\\
Since saturated log blow ups can be globalized on $\hat \X''$ (cf. \cite{Niz06} §4) we end up with a proper transform $\hat \X'$ and a projective birational map $\hat \X' \to \hat \X''$ whose composite to $\hat S'$ is  étale locally of the form (4) and (5).
Since $\hat \X'$ is dominating $\hat S'$, the corresponding map is flat (cf. \cite{Bou72} Ch.VI §3.6 Lem. 1).\\
Around a closed point $x \in \hat \X'$ in the very special fibre of $\Spec(A)$, the point $x$ corresponds in the local description to the maximal ideal $\fm=\pp+(X_1,\dots,X_s)$ in $B$, where $\pp$ is the maximal ideal of $A$. Localizing on $\fm$ then makes $A \to B$ flat by descent \cite{GW10} 14.12. Now the special choice of the monomials will be derived by iterative reduction of the overall degree of $M_1\cdots M_N$. 
First it is shown that we can choose the $M_k$ such that there is no pair $i\neq j$ with $M_i \mid M_j$.\\
In the valuation ring either $t_j \mid t_i$ or $t_i \mid t_j$ holds.
In the first case we had $t_i=t_i't_j$ with a $t_i' \in A$ and $M_j=M_j'M_i$ with a monomial $M_j'$. We then have:
\[ t_j(t_i'M_j'-1) \equiv t_jt_i'M_j'-t_j \equiv t_iM_j'-t_j \equiv M_iM_j-t_j \equiv M_j-t_j \equiv 0. \]
By flatness $t_j$ is not a zero divisor on $B'$.
 So $t_i'M_j'-1 \in \left(t_1-M_1,\cdots, t_N-M_N \right)$, which is a contradiction to the fact that the parameters $t_k$ can not be $A$-linear combined to $-1$. So the case $t_i \mid t_j$ can not occur.\\
In the second case we have $t_j=t_j't_i$ with a $t_j' \in A$ and $M_j=M_j'M_i$ as before. In this case we have:
\[ t_i(t_j'-M_j') \equiv t_it_j'-M_iM_j' \equiv t_j-M_j \equiv 0. \]
By flatness $t_i$ is not zero divisor on $B'$. So $M_j' \equiv t_j'$ and we can replace the relation $t_j-M_j$ by $t_j'-M_j'$. Then $t_1,\dots, t_j',\dots, t_N$ still is an ordered system of parameters and the overall degree drops. Repeating this process will give a representation as in (4) and (5) with no $M_i$ dividing any $M_j$, $j \neq i$.\\
Now assume there still were a monomial $M_i$ having a common variable with a $M_j$, $j \neq i$.  
Consider the monomial $W=\gcd(M_i,M_j) \neq 1$. Then $M_i=M_i'W$ and $M_j=M_j'W$ with coprime monomials $M_i', M_j'$.
W.l.o.g. assume the case $t_j=t_j't_i$ (otherwise swap $i$ and $j$). Then we get:
\[ t_i(M_j'-t_j'M_i') \equiv (M_i'W)(M_j'-t_j'M_i') \equiv M_i'(M_j-t_j'M_i') \equiv M_i'(t_j-t_j't_i) \equiv 0.  \]
Again by flatness $t_i$ is not a zero divisor on $B'$ and we get that
$M_j'-t_j'M_i'$ lies in the ideal $\left(t_1-M_1,\cdots, t_N-M_N \right)$.
Again since the parameters $t_k$ can not $A$-linear combined to $1$ there must be at least one index $k$ such that $M_k \mid M_j' \mid M_j$. Since $\deg M_j' < \deg M_j$ we have $k \neq j$. But this was ruled out by the previous arguments.
So the monomials $M_k$ can be assumed to be pairwise coprime. \\
Furthermore, because the torsion of the group
\[\coker((P')^\gp \to (Q'')^\gp)\iso\coker((P')^\gp \to (Q')^\gp)\iso\coker(P^\gp \to Q^\gp)\]
is finite of order prime to $p$ and the $M_k$ are pairwise coprime
the special form of the exponents in the monomials follows.\\
Throughout this construction the closed subset $\hat Z'$ is locally given by the vanishing of a product of generators of the monoids, i.e. w.l.o.g. by a zero set $V(X_1\cdots X_r)$ in (5).
\[ \xymatrix{
\Spec(C)\ar[r]\ar^-\et[d]&\hat\X' \ar[r] \ar^-b[rrrd] \ar^-{\hat f'}[rddd]&
\hat S'\x_S \X' \ar^-{f'\x_S\hat S'}[ddd]\ar[r]&\hat S \x_S\X'\ar^-{b\x_SS'}[dr] \ar[rrr] \ar\ar^-{f'\x_SS'}[ddd] &&& \X' \ar_-b[ld] \ar^-{f'}[ddd]\\
\Spec(B)\ar[dd]&&&& \hat\X \ar^-{\hat f}[d] \ar[r]& \X  \ar^-{f}[d] \\
&&&& \hat S \ar[r] & S  \\
\Spec(A) \ar[rr]&&\hat S' \ar[r] \ar_-{a}[rru]&\hat S \x_S S' \ar_-{a \x_SS'}[ru] \ar[rrr] &&& S' \ar^-a[lu]
}\]
\end{proof}

For later induction arguments we needed the degrees of the $p$-alterations in \ref{alteration} to be $1$, but this is not even clear for relative curves. Instead we will prove the following resolution of singularities result for curves $X$ with ordinary double points concerning the existence of a simple normal crossing model $\X$. We then can inductively take models of $\X_{s,\rd}$ and so on.

\begin{Thm}
\label{sncmr-curves}
Let $X=\bigcup_i X_i$ be a scheme of $\dim X =1$ separated and of finite type and with simple normal crossings over the field of fractions $K$ of an excellent henselian discrete valuation ring $\O_K$ (i.e. the irreducible components $X_i$ are regular and meet each other transversally) and let $W \ins X$ be a finite set of closed points disjoint from the singular points of $X$.
Then there is a scheme $\X=\bigcup_i \X_i$ over $\O_K$ with $X \iso_K \X_\eta$ such that $\X$ and
\[\X^{[k]} := \coprod_{i_0 < i_1 < \dots < i_k} \X_{i_0} \x_\X \X_{i_1} \x_\X \dots \x_\X \X_{i_k}\]
for all $k \ge 0$ are flat, separated and of finite type over $\O_K$ with quasi-semi-stable reduction such that the  $\X^{[k]}$ are all regular and
$\ol W \cup \X_{s,\rd}$ and $\ol W \cup \X_{s,\rd}^{[k]}$ are simple normal crossing divisors, where $\ol W$ is the Zariski closure of $W$ in the corresponding scheme. If $X$ is proper/projective over $K$, then $\X$ can be chosen to be proper/projective over $\O_K$.
\begin{proof}
Let $s,\eta$ be the special and the generic point of $S=\Spec(\O_K)$, resp..
Let $X_i$ be the (regular) irreducible components of $X$ and $Z_i$ be the finite subscheme of intersection points on $X_i$ and $W_i=W \cap X_i$. By \cite{CJS09} Cor. 0.4 every $X_i$ has a regular model $\X_i'$. Now consider the divisor 
$D_i:=\X_{i,s}' \cup (\ol{Z_i\cup W_i}^{\X_i'})_\rd $ on $\X_i'$. 
Again using resolution of singularities \cite{CJS09} Cor. 0.4 + remark (!) there are blow-up morphisms $f_i: \X_i \to \X_i'$, such that $\X_i$ is regular, $f_i^{-1}(D_i)$ is a simple normal crossing divisor and $f_i$ is an isomorphism outside the singularities of $D_i$ (not only of $\X_i'$). In particular, we have that $\X_{i,\eta} \iso \X_{i,\eta}' \iso X_i$, and that $(\ol{Z_i}^{\X_i})_\rd$ is regular, i.e. for $z \in Z_i$ we have $\ol{\{z\}}^{\X_i}_\rd\iso \Spec(B)$, where $B$ is the unique henselian discrete valuation ring extension in $\k(z)$ (cf. \cite{Mil80} Thm. 4.2). This uniqueness result allows us to glue (cf. \cite{Scw05} Cor. 3.9) the $\X_i$ along the closed subschemes $(\ol{Z_i}^{\X_i})_\rd$ to get a scheme $\X$ with the wanted properties. Also see \cite{DasD}.
\end{proof}
\end{Thm}

\section{Cohomological Hasse principle}

The main source for higher local fields and their class field theory are 
\cite{Kat7982}, \cite{Kat00}, \cite{Fes91a}, \cite{Fes91b}, \cite{Fes91c}, \cite{Fes91d}, \cite{Fes93}, 
\cite{Fes95}, \cite{Fes96}, \cite{FK00}, \cite{FV02} and \cite{Ras95} or \cite{Mor12} for a summary.

\begin{Def}[Higher local fields]\hfill
\label{hlf-def}
\begin{enumerate}
\item
  \begin{itemize}
    \item A {\emph {$0$}-local field} $F_0$ is defined to be a {\emph finite field}.
    \item A {\emph {$N$}-local field} $F_N$, $N \ge 1$, is a complete discrete valuated field whose residue field $F_{N-1}$ is a $(N-1)$-local field.
    \item So, a {\emph {$1$}-local field} $F_1$ is the same as a {\emph local field}.
   \end{itemize}
$ $ \\[-14pt]  
\noindent A \emph{higher local field} $F$ is a $N$-local field $F_N$ for some $N \ge 0$, and we will also address it with the tupel of its consecutive residue fields $(F=F_N, F_{N-1}, \dots, F_0)$. We usally put $p=\ch(F_0)$.
\item For every $1 \le m \le N$ let $\O_m=\O_m^{(1)}$ be the complete discrete valuation ring of $F_m$ 
and put $S_m=\Spec(\O_m^{(1)})$ with its generic point $\eta_m=\Spec(F_m)$ and special point $s_m=\Spec(F_{m-1})$. With $\eta_0=\Spec(F_0)$ we then have $s_{m+1}=\eta_m=\Spec(F_m)$ for all $0\le m\le N$.\\
\item Let $\Cat_m=\Sch_\sft(S_m)$ be the category of all schemes separated and of finite type over $S_m$ and put $\Cat_0=\Sch_\sft(F_0)$.
\item For $2\le i \le m$ inductively put 
$\O_m^{(i)}=\{x \in \O_m^{(1)}\;|\; \ol x \in \O_{m-1}^{(i-1)} \}$. 
Then $\O=\O_N^{(N)}$ is called the \emph{rank-$N$-valuation ring} of $F$. It is known (cf. \cite{Mor12}) to be a henselian valuation ring of dimension $N$, the residue fields of its ideals $0=\pp_N \subsetneq \pp_{N-1} \subsetneq \cdots \pp_0$ are exactly given by $(F_N, F_{N-1}, \dots, F_0)$. Its value group is $F^\x/\O^\x \iso \Z^N$ (lexicographically ordered). Put $\hat S = \Spec(\O_N^{(N)})$. 
\item For an irreducible scheme $f:\X \to S_m$ essentially of finite type over $S_m$ we define
\[ \dim_{S_m}\X := \trdeg\left(K(\X_\rd):K(\ol{f(\X)}_\rd)\right) + \dim \ol{f(\X)}, \] 
where $\ol{f(\X)}$ denotes the topological closure of the image of $\X$ in $S_m$ and $\dim$ the usual Krull dimension in the Zariski topology. For an integer $a \in \Z$, furthermore, put
\[\X_{(a)}:=\{ x \in \X \;|\; \dim_S\ol{\{x\}}=a  \}.\]
Note, for $x \in \X$ we have 
\[\dim_S\ol{\{x\}} = \trdeg(\k(x):\k(f(x)))+\dim \ol{\{f(x)\}}.\]
\end{enumerate}
\end{Def}

\begin{Rem}
\label{ext-val-rings}
Let $F_N$ be a $N$-local field and $F'_N|F_N$ be a finite extension.
The valuation rings in the construction of $F'_N$ are then obtained inductively as follows:
For $m=N, N-1, \dots, 1$ the integral closure of $\O_m^{(1)}$ in $F_m'$ is a complete discrete valuation ring ${\O'}_m^{(1)}$, which is finite (!) over $\O_m^{(1)}$ since $F'_m|F_m$ is finite and complete discrete valuation rings are excellent. The corresponding residue field extensions $F'_{m-1}|F_{m-1}$ are then also finite and the induction can continue.\\
Furthermore, the rank-$k$-valuation rings ${\O'}_m^{(k)}$ of $F_m'$ are then the integral closure of $\O_m^{(k)}$ in $F_m'$.
Indeed, since $\O_m^{(k)}$ is henselian its integral closure in $F'_m$ is a valuation ring, which is dominated by ${\O'}_m^{(k)}$. So they must coincide.
\end{Rem}

\begin{DefLem}[Étale homology theory]
\label{etale-hom-def}
Fix $1\le m \le N$ and an integer $n \in \Z$ not divisible by $\ch(F_{m-1})$.
 For $f: \X \to S_m$ in $\Cat_m$ we then define the \emph{étale homology groups} by
\[  \H_a^{(m)}(\X,\Z/n):=\H_a(\X|S_m,\Z/n):= \H^{m+2-a}_\et(\X,Rf^!\Z/n(m)_{S_m}), \]
where $\Z/n(b)_{S_m}$ is the sheaf $\mu_n^{\otimes b}$ on $S_m$ 
and $Rf^!$ is the right adjoint of $Rf_!$ defined in \cite{SGA4.3} XVIII, 3.1.4 or \cite{Fu11} §8.4. \\
Then by \cite{JS03}, \cite{KS10} or \cite{JSS09} this defines a homology theory $(H_a^{(m)})_{a \in \Z}$ on $\Cat_m$, the so called \emph{étale homology theory} over $S_m$, i.e. the $H_a^{(m)}$ are covariant functorial for proper maps
and contravariant functorial for open immersions in $\Cat_m$, and for every closed immersion in $\Cat_m$ there is a localization sequence which is functorial for open immersions and proper maps in $\Cat_m$.
\end{DefLem}

\begin{Prp} The étale homology theories have the following properties:
\begin{enumerate}
\item[\bf CMP:] When restricted to $\Sch_\sft(\eta_m)=\Sch_\sft(s_{m+1})$
the étale homology theories satisfy
\[ \H^{(m)} \iso \H^{(m+1)}[1]\]
in a natural way.
\item[\bf LVL:] 	$(\H_a^{(m)})_{a \in \Z}$ is \emph{leveled above $e=0$}, i.e. 
\[\H_a^{(m)}(\X)=0,\]
	whenever $\X$ is an affine, regular and connected scheme in $\Cat_m$ and 
	$a < \dim_{S_m}\X$. 
\item[ \bf NSS:] There is a first quadrant \emph{niveau spectral sequence} of homological type
\[ E^{(m),1}_{r,q}(\X,\Z/n)= \bigoplus_{x \in \X_{(r)}}\H^{m+r-q}_{\Gal}(\k(x),\Z/n(m+r-1)) \Longrightarrow \H_{r+q}^{(m)}(\X,\Z/n),\]
which is covariant for proper maps and contravariant for open immersions in $\Cat_m$. Furthermore, $E^{(m),1}_{r,q}(\X,\Z/n)$ is supported in the area $0 \le r \le \dim_{S_m}\X$ and $0 \le q \le m+r$.
\item[ \bf KH:] The \emph{Kato homology groups}
\[ \KH_a^{(m)}(\X,\Z/n):=E^{(m),2}_{a,0}(\X,\Z/n),\]
define a homology theory $(\KH_a^{(m)})_{a \in \Z}$ on $\Cat_m$ and the edge morphisms 
\[\epsilon^a_\X: \H_{a}^{(m)}(\X) \longrightarrow E^{(m),2}_{a,0}(\X)=\KH_a^{(m)}(\X)\]
induce a morphism of homology theories $\epsilon: \H^{(m)} \longrightarrow \KH^{(m)}$.

\item[\bf DSS:] For every $\X=\bigcup_{i \in I}\X_i$ in $\Cat_m$ covered by a finite ordered number of closed subschemes $\X_i \ins \X$, we have natural bounded convergent \emph{descent spectral sequences} of homological type
\[ \xymatrix{
 E^{(m),1}_{a,b}=\H_{b}^{(m)}(\X^{[a]}) \ar@{=>}[r] \ar^-\epsilon[d]& \H_{a+b}^{(m)}(\X) \ar^-\epsilon[d], \\
 'E^{(m),1}_{a,b}=\KH_b^{(m)}(\X^{[a]})  \ar@{=>}[r]& \KH_{a+b}^{(m)}(\X),}\]
which are naturally compatible via the edge morphisms, where we put
\[\X^{[k]} := \coprod_{i_0 < i_1 < \dots < i_k} \X_{i_0} \x_\X \X_{i_1} \x_\X \dots \x_\X \X_{i_k}\]
to be disjoint union of $k$-fold intersections of the corresponding subschemes, $k \ge 0$.

\item[\bf PB:] \emph{Pullback}: For any embeddable morphism $f : \X' \to \X$ (i.e. factors through a closed subscheme of a smooth one) between connected regular schemes in $\Cat_m$ with $\dim_{S_m}(\X') = \dim_{S_m}(\X)$ there exist functorial pullback morphisms $f^*$ on $H_a^{(m)}$ and $\KH_a^{(m)}$ compatible with the edge morphisms and extending the pullback for open immersions, and such that  $f_*\circ f^*$ is the multiplication by $\deg(\X'|\X)$ on $\KH_a^{(m)}(\X)$ whenever $f$ is proper and dominant, $a \in \Z$.

\item[\bf LC:] \emph{Lefschetz condition over finite fields}: For every irreducible regular $X$ projective over the finite field $F_0$ together with a simple normal crossing divisor $Y$ which contains an irreducible component $Y'$ such that $X \sm Y'$ is affine, the composite map
\[ H_{a}^{(1)}(X \sm Y) \stackrel{\epsilon}{\to} \KH_{a}^{(1)}(X\sm Y) \to \H_a(\Gamma_{X\sm Y})\]
is an isomorphism for all $a \le \dim X$, and injective for $a=\dim X+1$,
 where the last group is the weight homology group attached to $H^{(1)}$ (cf. \cite{KS10} §2).

\item[\bf SAL:] \emph{Strong affine Lefschetz}: For every connected regular $\X$ which is flat and projective over $S_m$ with quasi-semi-stable reduction and every $S_m$-flat divisor $W \ins \X$ such that $\X_{s_m,\rd} \cup W$ is a simple normal crossing divisor on $\X$ and such that 
there is an irreducible component $W' \ins W$ such that $\X\sm W'$ is affine, we have $\H_a^{(m)}(\X \sm W)=0$ for all $a \le \dim_{S_m}\X$ (and not only for $a < \dim_{S_m} \X$).
\end{enumerate}
\begin{proof}
({\bf CMP}): Let $f: X \to S_m$ factor through $g: X \to \eta_m=s_{m+1}$. And consider the open immersion $j: \eta_m \to S_m$ and closed immersion $i: s_{m+1} \to S_{m+1}$ and $f':=i \circ g : X \to S_{m+1}$. Then we have $Rf^!=Rg^! \circ j^*$ and ${Rf'}^!=Rg^!\circ i^!$ and therefore
\[\begin{array}{lll}
\H_{a-1}^{(m+1)}(X,\Z/n) &= & \H^{(m+1)+2-(a-1)}_\et(X,Rg^! \circ i^! \Z/n(m+1)_{S_{m+1}})\\
	&= & \H^{m+4-a}_\et(X,Rg^! \Z/n(m+1-1)_{s_{m+1}}[-2])\\
	&= & \H^{m+2-a}_\et(X,Rg^! \Z/n(m)_{\eta_m})\\
  &= & \H^{m+2-a}_\et(X,Rg^! \circ j^* \Z/n(m)_{S_m}) \\
  &= & \H_a^{(m)}(X,\Z/n),
\end{array}\]
using the absolute purity theorem \cite{Fuj02} for $i^!$.\\
({\bf LVL}): If $\X$ is regular, connected and $f$ is an embedable morphism and $\dim_{S_m} \X =d+1$, then by Gabber's absolute purity theorem \cite{Fuj02} and \cite{Fu11} §8.4, we have
\[ Rf^!\Z/n(m)_{S_m} \iso \Z/n(m+d)[2d]_\X,  \]
and therefore
\[ \H_a^{(m)}(\X,\Z/n) \iso \H^{2d+m+2-a}_\et(\X,\Z/n(m+d)_\X).  \tag{O}\]
Also see \cite{SS10} Lem. 1.8. If, furthermore, $\X$ is affine, we have 
\[\cd \X = \dim_{S_m} \X + \cd F_{m-1} =d+1+m\]
by Gabber's affine Lefschetz theorem (cf. \cite{Ill03} Thm. 2.4). So $\H_a^{(m)}(\X)$ vanishes for $a < \dim_{S_m} \X$ for $\X$ affine.\\
({\bf NSS}): By \cite{BO74} §3 and \cite{JS03} §2 there always is a functorial niveau spectral sequence
\[ E^1_{r,q}(\X)= \bigoplus_{x \in \X_{(r)}} \H_{r+q}(x) \quad\Longrightarrow\quad \H_{r+q}(\X),  \]
with $\H_a(x) := \inlim_{x \in V \ins \ol{\{x\}}} \H_a(V)$
for $x \in \X$, where the direct limit is taken over all open non-empty subschemes $V \ins \ol{\{x\}}$ and all (descending) inclusion maps. The direct limit can then in our special case be computed by (O). Note that if $x \in \X_{(r)}$ then $\cd \kappa(x) \le m+r$ (cf. \cite{Ser02} Ch.II §4-6). \\
({\bf KH}): This is also a general construction. The \emph{Kato complex of $\X$ w.r.t. $\H$} is definited to be the bottom line of the niveau spectral sequence, i.e. 
 \[C(\X):\ E^1_{0,e}(\X) \stackrel{d^1}{\longleftarrow} E^1_{1,e}(\X) \stackrel{d^1}{\longleftarrow} E^1_{2,e}(\X) \stackrel{d^1}{\longleftarrow} \cdots \stackrel{d^1}{\longleftarrow} E^1_{d,e}(\X),\]
for $e=0$, where $d=\dim_{S_m}\X$, $d^1$ are the differentials of the $E^1$-layer, and $E^1_{a,e}(\X)$ is put in degree $a$. The Kato complex is then functorial like $H$. For a closed immersion $i: Y \climm \X$ with open complement $j: V \opimm \X$ we have a functorial exact sequence (by definition of $\dim_{S_m}$):
 \[0 \longrightarrow C_a(Y) \stackrel{i_*}{\longrightarrow} C_a(\X) \stackrel{j^*}{\longrightarrow}
 C_a(V) \longrightarrow 0.\]
Its homology groups then give the Kato homology groups and the functorial long exact localization sequence.\\
({\bf DSS}): The bottom spectral sequence can be obtained as the hypercohomology spectral sequence of the obvious double complex of the $E^1$-layer of the niveau spectral sequence using the naive filtration (cf. \cite{Wei95} §5.6.2). 
For the upper spectral sequence first let $i^{[r]}: \X^{[r]} \to \X$ be the canonical map. Then for every étale sheaf $F$ on $\X$ there is an exact complex
\[ 0 \to F \to i^{[0]}_*i^{[0]*}F \to i^{[1]}_*i^{[1]*}F \to \dots \to i^{[n-1]}_*i^{[n-1]*}F \to 0, \]
where the differential maps are given by the alternating sums of the restriction maps. This can be proven by induction on the number of closed subschemes. For $n=1$ this is trivial. The induction step is achieved by the exact sequence of sheaves
\[ 0 \to F \to i_*i^*F \oplus j_*j^*F \to k_* k^*F \to 0,  \]
where $i: \hat \X \climm \X$ with $\hat \X = \bigcup_{i=1}^{n-1} \X_i$ and $j: \X_n \climm \X$ and $k: \hat \X \cap \X_n \climm \X$. The exactness is immediately seen on the stalks.\\
Now let $I$ be an injective representative of $Rf^!\Z/n(m)_{S_m}$, where $f: \X \to S_m$ is the structure morphism. 
Then for any étale sheaf $F$ on $\X$ the complex
\[\dots \Hom(i^{[\bullet]}_*i^{[\bullet]*}F, I) \to \Hom(F,I) \to 0 \]
is exact by the sequence above and the injectivity of $I$. By adjunction 
\[ \Hom(i^{[k]}_*i^{[k]*}F,I) = \Hom(F, i^{[k]}_*Ri^{[k]!}I),  \]
we get an exact complex 
\[\dots \Hom(F, i^{[\bullet]}_*Ri^{[\bullet]!}I) \to \Hom(F,I) \to 0 \]
for any étale sheaf $F$. Therefore the complex
\[ \dots \to i^{[1]}_*Ri^{[1]!}I \to i^{[0]}_*Ri^{[0]!}I \to I \to 0 \]
is exact. From this we get a double complex. The spectral sequence is then obtained as the hypercohomology spectral sequence, cf. \cite{Wei95} Ch. 5, for the naive filtration.\\
({\bf PB}): The pullback construction can be found in \cite{KS10} §4.\\
({\bf LC}): The Lefschetz condition was proved in \cite{KS10} Thm. 3.5 (also cf. \cite{SS10}, \cite{JS08}) using weight arguments (cf. \cite{Del80}).\\
({\bf SAL}): Let $W=\bigcup_{i=1}^rW_i$ be the irreducible components and $W_1$ such that $\X \sm W_1$ is affine. We do induction on $d+1=\dim_{S_m} \X$ and $r$:
 For $d > 0$ and $r >1$ the exact localization sequence gives:
\[ 0\stackrel{ r-1}{=}\H_a^{(m)}(\X \sm \bigcup_{i=1}^{r-1}W_i) \to 
\H_a^{(m)}(\X \sm W) \to 
 \H_a^{(m)}(W_r \sm \bigcup_{i=1}^{r-1}(W_i \cap W_r)\stackrel{d}{=}0,  \]
since $W_r \sm (W_1 \cap W_r)$ is affine and $\dim_{S_m}W_r=d$. So 
$\H_a^{(m)}(\X \sm W)=0$ and we are reduced to the case $d=0$ or $r=1$.
For $d=0$ then $\X$ is finite over $S_m$ and therefore affine itself.
For $r=1$ and $U=\X \sm W_1$ the base change \ref{base-change} gives us
\[ \begin{array}{lll}
\H_{d+1}^{(m)}(U,\Z/\ell^r) &\iso& \H^{d+1+m}_\et(U,\Z/\ell^r(m+d+1)) \\
	&\iso& \H^{m+d+1}_\et(U_{s,\rd},\Z/\ell^r(m+d+1)),
\end{array} \]
which vanishes by affine Lefschetz theorem \cite{Mil80} VI Thm. 7.2 and 
$m+d+1 > m+d=\cd(F_{m-1})+\dim (U_{s,\rd})$.
\end{proof}
\end{Prp}

\begin{Lem} 
\label{base-change}
Let $\O$ be a henselian discrete valuation ring, $S=\Spec(\O)$ with generic point $\eta$ and special point $s$ and let $\ell$ be a prime invertible on $S$. Let $\X$ be regular, pure dimensional scheme flat and projective over $S$ with quasi-semi-stable reduction and $W \ins \X$ a flat sncd such that $U:=\X \sm W$ is affine and $\X_{s,\rd} \cup W$ is a sncd of $\X$. 
Then the pull-back map
\[  \H_\et^q(U,\Z/\ell^r(b)) \longrightarrow  \H_\et^q(U_{s,\rd},\Z/\ell^r(b)) \]
is an isomorphism for any integers $r \ge 1$, $q,b \in \Z$.
\begin{proof}
This is a generalization of \cite{RZ82} Satz. 2.19 by using Gabber's purity theorem \cite{Fuj02} Thm. 2.1.1. The proof is given in \cite{SS10} Lem. 3.4. 
\end{proof}
\end{Lem}

\begin{Thm}[Kato's conjecture about the cohomological Hasse principle]
\label{etale-hasse}
Let $(F_N, \dots, F_0)$ be an $N$-local field with $N\ge1$, $p=\ch(F_0)$ and its complete discrete valuation ring $\O_N$, $S=\Spec(\O_N)$.
Let $n$ be an integer not divisible by $p$. Let $f: \X \to S$ be a flat and proper map from a regular scheme $\X$ to $S$, $\dim_S\X=d+1$. Then we have:
\[ \textbf{KC}(N,d+1):\qquad \KH_a^{(N)}(\X,\Z/n)=0\qquad \forall a \in \Z. \]
That means that the Kato complex from \cite{Kat86}:
\[\begin{array}{llllll} 
0 &\longleftarrow&\ds \bigoplus_{x \in \X_{(0)}} \H^N(\k(x),\Z/n(N-1))& \longleftarrow &\ds\bigoplus_{x \in \X_{(1)}} \H^{N+1}(\k(x),\Z/n(N))\\
\cdots& \longleftarrow &\ds\bigoplus_{x \in \X_{(d+1)}} \H^{N+d+1}(\k(x),\Z/n(N+d))&\longleftarrow &0 
\end{array}\]
is exact.
\begin{proof} 
We can assume $\X$ to be integral with $\dim_S\X=d+1$ and $n=\ell^r$ for a prime $\ell \neq p$. We will then do induction on $N+d+a$ and reduce to varieties over finite fields.\\
For $d=0$ the scheme $\X$ is a finite extension of $S$. Let $\eta$ be the generic point and $s$ the special point of $\X$. Then the Kato complex is the middle two terms of the exact localization sequence
\[ \H_1^{(N)}(\X) \to \H_1^{(N)}(\eta) \to \H_0^{(N)}(s) \to \H_0^{(N)}(\X), \]
where the first and last term here vanishes by {\bf SAL}. So $\KH_a^{(N)}(\X)=0$ for $a=0,1$ and for all other $a$ trivially by dimension.\\
 For $d >0$ let $\alpha \in \KH_a^{(N)}(\X)$. We suppress the index $(N)$ in the following $E$-terms.\\
Since $E^{r}_{s,t}=0$ for $s <0$ and 
\[E^{r+1}_{a,0}\iso\ker \left(E^r_{a,0} \stackrel{d^r}{\longrightarrow} E^{r}_{a-r,r-1}\right),\]
the edge morphism for every open subscheme $U \ins \X$ factors as follows:
\[\H_a^{(N)}(U) \srj E^\infty_{a,0}(U)\iso E^{a+1}_{a,0}(U)  \cdots \inj E^{3}_{a,0}(U) \inj E^{2}_{a,0}(U) = \KH_a^{(N)}(U). \]
By successively cutting out the support of $d^2(\alpha)$, $d^3(\alpha)$, \dots, $d^a(\alpha)$ one can find a closed subscheme $W \ins \X$ of dimension $\dim_S W \le a-2$, $U:=\X\sm W$, such that $\alpha|_U$ lies in the image of $\H_a^{(N)}(U)$ under the edge map. Let $\beta \in \H_a^{(N)}(U)$ be a preimage. \\
By applying \ref{alteration} $\X$ can be assumed to be projective over $S$, let say $f$ factors through $\X \inj \P^m_S \to S$. 
Now let $\O_N^{(N)}$ be the rank-$N$-valuation ring of $F_N$ and $\hat S =\Spec(\O_N^{(N)})$ and $\hat \X$ be the closure of $\X$ in $\P^m_{\hat S}$ and $\hat Z$ the one of $W$. Applying \ref{alteration} there is a $p$-alteration  $b: \hat \X' \to \hat \X$ such that $\hat \X'$ has poly-quasi-semi-stable reduction over a valuation ring of a finite $p$-extension of  $\Quot(\O_N^{(N)})$ which is integral over $\O_N^{(N)}$.
The reduced structure of the pullback over $S_m$  will then be projective 
as explained in \ref{alt-rem}.
 Furthermore, we have a simple normal crossing divisor $\hat Z'' =  b^{-1}(\hat Z) \cup \hat Y'$, where $\hat Y'$ is the preimage of the codimension one closed subset of $\hat S$. 
After applying the Bertini theorem \ref{bertini} to $\hat \X'$ the $\hat Z''$ contains an irreducible component $H$ whose complement in $\hat \X'$ is affine and the very special fibres of all multiple intersections are cut transversally by $H$.\\
$\X':=(\hat \X' \x_{\hat S} S)_\rd$ is then regular by \ref{pqss-red-rem} with quasi-semi-stable reduction and simple normal crossing divisor $W':=(\hat Z''\x_{\hat S} S)_\rd$ and $\pi:=b_S: \X' \to \X$ is a projective $p$-alteration. 
Put $U'=\X'\sm W'$, $U''=\X'\sm \pi^{-1}(W)$ and consider the following diagram (all over $S=S_N$), which is commutative by the {\bf PB}-condition:
\[\xymatrix{\KH_{a+1}^{(N)}(U') \ar@{->>}^-{\delta}[rd] \\
\H_{a+1}^{(N)}(U')\ar^-{\epsilon}[u]\ar@{-->>}[r]&\KH_a^{(N)}(W') \ar^-{=0}[rd]\\
\alpha \ar@{|->}[d]\ar@{}[r]|{\in} &\KH_a^{(N)}(\X) \ar@{^(->}^-{\pi^*}[r] \ar[d] & \KH_a^{(N)}(\X') \ar[d]\ar@{^(->}^-{j^*}[rd] \ar@{->>}^-{\pi_*}[r]& \KH_a^{(N)}(\X) \\
\alpha|_U \ar@{}[r]|{\in} &\KH_a^{(N)}(U) \ar^-{\pi^*}[r] & \KH_a^{(N)}(U'') \ar[r] & \KH_a^{(N)}(U') \ar^-{\delta}[rd]\\
\beta \ar@{|->}[u]\ar@{}[r]|{\in} &\H_a^{(N)}(U) \ar^-{\epsilon}[u]\ar^-{\pi^*}[r] & \H_a^{(N)}(U'') \ar[r]\ar^-{\epsilon}[u] & \H_a^{(N)}(U')\ar^-{\epsilon}[u] \ar@{^(-->}[r] &\KH_{a-1}^{(N)}(W')
}\]
Now assume that the upper dotted line is surjective and the lower dotted line is injectiv for $0 \le a \le d+1$: 
Since the diagonal sequence is exact, the element $\pi^*(\alpha)$ maps to $0 \in \KH_{a-1}^{(N)}(W')$. So the same holds for 
$\pi^*(\beta)|_{U'}$. By assumption the dotted bottom line is injective, so $\pi^*(\beta)|_{U'}=0$ and then also $j^*(\pi^*(\alpha))=0$. The fact that the top dotted line is surjective by assumption implies that $j^*$ is injective, so already $\pi^*(\alpha)=0$.
Furthermore, the compositum $\pi_*\circ \pi^*: \KH_a^{(N)}(\X) \to \KH_a^{(N)}(\X)$ is multiplication by the degree by {\bf PB}, which is prime to $\ell$. Because $\KH_a^{(N)}(\X)$ is $\ell$-primary torsion, it follows that the map $\pi_*\circ \pi^*$ is an isomorphism and $\pi^*$ is injective, so $\alpha=0$ and the claim follows.\\
For the statement of the dotted lines let 
$\hat Z'$ be the flat part of $\hat Z''$ and $\hat U':=\hat \X' \sm \hat Z'$.
$Z':=\hat Z'_{S_N}$ is then the flat part of $W'= \hat Z'_{S_N} \cup \hat \X'_{N-1}$ and
\[ U'=\X'\sm W'=\hat U'_N,\qquad  W'\sm Z'= \hat U'_{N-1}, \qquad \X' \sm Z'=\hat U'_{S_N}, \]
where the index indicates the basechange with $F_{N}$, $F_{N-1}$, $S_N$, resp., in the reduced structure.\\
For $m=N$ consider the (anti-)commutative diagram, where the second column is the dotted line from above:
\[\xymatrix{
\H_a^{(m)}(\hat U'_{S_m}) \ar^-{j*}[r]\ar^-\epsilon[d]& \H_a^{(m)}(\hat U'_m) \ar^-{\delta}[r]\ar^-\epsilon[d]&
 \H_{a-1}^{(m)}(\hat U'_{m-1}) \ar^-\epsilon[d]\ar^-{i_*}[r]& \H_{a-1}^{(m)}(\hat U'_{S_m})\ar^-\epsilon[d]\\
\KH_a^{(m)}(\hat U'_{S_m}) \ar^-{j*}[r]\ar^-\delta[d]& \KH_a^{(m)}(\hat U'_m) \ar^-{\delta}[r] \ar^-\delta[d]&
\KH_{a-1}^{(m)}(\hat U'_{m-1})\ar@{=}[d]\ar^-{i_*}[r] & \KH_{a-1}^{(m)}(\hat U'_{S_m})\ar^-\delta[d]\\
\KH_{a-1}^{(m)}(\hat Z'_{S_m})\ar^-{i_*}[r]&
\KH_{a-1}^{(m)}(\hat Z'_{S_m} \cup \hat \X'_{m-1}) \ar^-{j*}[r]&
\KH_{a-1}^{(m)}(\hat U'_{m-1})\ar^-{\delta}[r] & \KH_{a-2}^{(m)}(\hat Z'_{S_m}).
}\]
Since $\dim_SW'=d$ we have $\KH_{d+1}^{(N)}(W')=0$ and the surjectivity of
\[ H_{d+2}^{(N)}(\hat U'_N) \to H_{d+2}^{(N)}(\hat U'_N) \to
\KH_{d+1}^{(N)}(\hat Z'_{S_N} \cup \hat \X'_{N-1})= \KH_{d+1}^{(N)}(W')=0 \]
 is trivial. \\
Now by {\bf SAL} and the previous usage of Bertini we have 
\[\H_q^{(N)}(\hat U'_{S_N})=0 \qquad  \forall q \le d+1,\]
and the upper $\delta$ is an isomorphism for $a \le d+1$. \\
Separating $\hat Z'_{S_N}$ into its irreducible components, the induction hypothesis {\bf KC}(N,d) and by {\bf DSS} we have
\[ \KH_{q}^{(N)}(\hat Z'_{S_N}) \stackrel{r \ge 0}{\Longleftarrow} 
\KH_{q-r}^{(N)}((\hat Z'_{S_N})^{[r]})=0,\]
and ergo $\KH_{q}^{(N)}(\hat Z'_{S_N}) = 0$ for all $q \in \Z$ and the lower $j^*$ is an isomorphism. \\
It is then enough to show that the middle right 
$\epsilon: \H_{a-1}^{(m)}(\hat U'_{m-1}) \to \KH_{a-1}^{(m)}(\hat U'_{m-1})$ 
is an isomorphism for $m=N$ and $a \le d+1$. 
By {\bf CMP} and {\bf DSS} we have for $m=N$:
\[\xymatrix{
  \H_{a-1}^{(m)}(\hat U'_{m-1}) \ar^-\epsilon[d] \ar@{=}[r]&  \H_{a}^{(m-1)}(\hat U'_{m-1}) \ar^-\epsilon[d] \ar@{<=}^{r \ge 0}[r]&  
	\H_{a-r}^{(m-1)}((\hat U'_{m-1})^{[r]}) \ar^-\epsilon[d]\\
 \KH_{a-1}^{(m)}(\hat U'_{m-1}) \ar@{=}[r]&  \KH_{a}^{(m-1)}(\hat U'_{m-1})  \ar@{<=}^{r \ge 0}[r]&  \KH_{a-r}^{(m-1)}((\hat U'_{m-1})^{[r]})
}\]
and an isomorphism on the right for fixed $a$ and all $r \ge 0$ would imply the isomorphism on the left. 
Note that by \ref{pqss-red-rem} $W'\sm Z'=\hat U'_{N-1}$ has simple normal crossings and so $(\hat U'_{N-1})^{[r]}=(W'\sm Z')^{[r]}$ is of pure dimension $d-r$.\\
By \ref{pqss-red-rem} the closure $\hat V$ of every irreducible component $V$ of $(\hat U'_{N-1})^{[r]}$ in $\hat \X'$ again has poly-quasi-semi-stable reduction over the rank-$(N-1)$-valuation ring of $F_{N-1}$. We then can consider the diagram for $m=N-1$ and $\hat V$ instead of $\hat \X'$. As before the bottom $j^*$ and the upper $\delta$ are isomorphisms.
Furthermore, the induction hypothesis {\bf KC}($N-1,d+1$) implies that $\KH_q(\hat V_{S_{N-1}})=0$  for every $q \in \Z$ and so the middle left $\delta$ is an isomorphism as well. So we can continue this process and end up over varieties over finite fields. It is then left to prove:\\
For every connected variety $X$ of dimension $d$ smooth and projective over a finite field with a simple normal crossing divisor $D$ which contains an irreducible component $D'$ such that $X\sm D'$ is affine the edge morphism
\[ \epsilon:\; \H_a^{(1)}(X \sm D) \to \KH_{a}^{(1)}(X \sm D) \]
is an isomorphism for all $a \le d$.\\
But this is implied by {\bf LC} (cf. \cite{KS10} Lem. 6.6, \cite{KS10} Cor. 6.3 and the LG-condition there).
And the proof is complete.
\end{proof} 
\end{Thm}

\begin{Cor}
\label{fibre-eq}
Let $(F_N, \dots, F_0)$ be a $N$-local field and $\O_N^{(N)}$ its rank-$N$-valuation ring, $\hat S = \Spec(\O_N^{(N)})$. Let $n \in \Z$ not be divisible by the residue characteristic $p = \ch(F_0)$.
Let $\hat \X$ be an integral scheme flat and proper over $\hat S$ with poly-quasi-semi-stable reduction with generic fibre of $\dim \hat \X_N =d$.  
Then for all $a \in \Z$ we have canonical isomorphisms of finite (!) groups
\[ \KH_{a+1}^{(N)}(\hat \X_N,\Z/n) \iso \cdots \iso \KH_{a+1}^{(m)}(\hat \X_m,\Z/n) \iso  \KH_{a}^{(1)}(\hat \X_0,\Z/n) \iso \H_a(\Gamma_{\hat \X_0},\Z/n), \]
where the last group ist the weight homology group of the very special fibre $Y:=\hat \X_{0,\rd}$ over the finite field (cf. \cite{KS10}), 
i.e. the homology groups of the \emph{dual complex} of $Y=\bigcup_i Y_i$ with irreducible components $Y_i$:
 \[ \G_Y:\; \dots \to (\Z/n)^{\pi_0(Y^{[a]})} \to (\Z/n)^{\pi_0(Y^{[a-1]})}  \to \dots,\]
 where the maps are induced by the alternating sums of the maps which leave out one component.\\
The isomorphisms are also true if $\hat \X_N$ is a proper simple normal crossing curve and $\hat \X_m$ are the reductions of inductively taken simple normal crossing models of \ref{sncmr-curves}.
\begin{proof}
Let $\O_m$ be the complete discrete valuation ring of $F_m$ and $S=\Spec(\O_m)$.
Then $\X:= \hat \X \x_{\hat S} S$ by the local discription \ref{pqss-red} is the union of its irreducible components $\X=\bigcup_{i \in I} \X_i$ such that for all $r \ge 0$ the schemes $\X^{[r]}_\rd$ are regular of some dimension and proper over $S$. So by {\bf DSS} and {\bf KC} \ref{etale-hasse} we have the vanishing of $\KH_{q}^{(m)}(\X)$ for all $q \in \Z$. 
Then by the exact localization sequence we get:
\[ 0=\KH_{a+1}^{(m)}( \X) \to \KH_{a+1}^{(m)}(\hat \X_{m}) \to \KH_{a}^{(m)}(\hat \X_{m-1}) \to \KH_{a}^{(m)}( \X)=0. \]
And with {\bf CMP} we have 
\[\KH_{a}^{(m)}(\hat \X_{m-1}) \iso \KH_{a+1}^{(m-1)}(\hat \X_{m-1})\] and we can go down to finite fields. The comparison with the weight homology groups is done in \cite{KS10} and they are finite. 
\end{proof}
\end{Cor}

\begin{Cor}
\label{finite-KH}
Let $F$ be an $N$-local field with final residue characteristic $p=\ch(F_0)$ and complete discrete valuation ring $\O_N$. Let $\X$ be a scheme separated and of finite type over $\O_N$ and $n \in \Z$ with $p \nmid n$.
Then the groups $\KH_a^{(N)}(\X,\Z/n)$ are finite for all $a \in \Z$.
\begin{proof}
By the localization sequence and a stratification by regular varieties one can reduce to the case where $\X$ is a regular and affine scheme over $F$ or its residue field. Taking a compactification of $X$ over the rank-$N$-valuation ring $\O$ of $F$ and using the alteration theorem \ref{alteration} one by {\bf PB} is reduced to the case where $\X = (\hat \X \sm \hat Z)_N$ and $\hat \X$ is of poly-quasi-semi-stable reduction with simple normal crossing divisor $\hat Z$ which is flat over $\O$.
From the localization sequence and \ref{fibre-eq} follows the claim.
\end{proof}
\end{Cor}

\begin{Rem}
\begin{enumerate}
\item 
For an $N$-local field $(F_N,\dots,F_0)$ let $(H_a^{(m)})_{a \in \Z}$ be  homology theories on $\Cat_m$ of $\ell$-primary torsion groups with $\ell \neq p =\ch(F_0)$, for $m=1,\dots,N$, satisfying {\bf CMP}, {\bf LVL}, {\bf DSS}, {\bf PB}, {\bf LC}, {\bf SAL}. Then the proof of \ref{etale-hasse} and \cite{KS10} Cor. 6.3 imply the vanishing of the corresponding Kato homology groups also for these homology theories. 
\item The results of this section also hold if one replaces the complete discrete valuation rings with excellent henselian discrete valuation rings where the final one has a finite residue field.
\end{enumerate}
\end{Rem}

\section{Motivic cohomology and cycle maps}

\begin{Rem}
\label{motivic-rem}
If $\X$ is a separated scheme of finite type over the spectrum of a field or a Dedekind ring $S$ let the motivic cohomology groups be defined by
\[ \H^r_M(\X,\Z/n(q)) := \CH^q(\X,2q-r;\Z/n),\]
Bloch's higher Chow groups with finite coefficients (cf. \cite{Blo86}, \cite{Lev04}, \cite{Gei04}, also see \cite{Voe02} for the field's case). 
If $\X$ is regular and $n$ is invertible on $\X$ 
 then there is an étale cycle map
\[ \rho_{\X,n}^{r,q}:\; \H^r_M(\X,\Z/n(q)) \to \H_\et^{r}(\X,\Z/n(q)) \]
constructed in \cite{Blo86}, \cite{GL01}, \cite{Lev04}, \cite{Gei04}.
By \cite{GL01},  \cite{SV00} and the now proven Bloch-Kato-conjecture (cf. \cite{Voe10b}, \cite{Voe10a}, \cite{SJ06}, \cite{HW09}) it is an isomorphism for $0 \le r \le q$ and injective for $r=q+1$. It is also an isomorphism for $q \ge d+c$, 
if $d:=\dim_S\X$ and $\cd(\k(s)) \le c$ for the closed points $s \in S$ (see below).\\
For any $k \in \Z$ and if $\X$, moreover, is smoothly embeddable over $S$ (e.g. quasi-projective),
 there is a commutative diagram of spectral sequences (\cite{JS08}, \cite{Lev01}, \cite{Blo94}):
\[\xymatrix{   
{^ME_{a,b}^1} \ar@{=}[r] \ar^-{\rho_{a,b}}[d]&
 \ds\bigoplus_{x \in \X_{(a)}} \H_M^{a-b}(\k(x),\Z/n(k+a)) \ar@{=>}[r] \ar[d]&
 \ar[d]\H_M^{2d-a-b}(\X,\Z/n(k+d)) \\
{^\et E_{a,b}^1} \ar@{=}[r] & 
\ds\bigoplus_{x \in \X_{(a)}} \H_\et^{a-b}(\k(x),\Z/n(k+a)) \ar@{=>}[r]& \H_\et^{2d-a-b}(\X,\Z/n(k+d)).
}\]
$\rho_{a,b}$ is an isomorphism for $b \ge -k$ by the Bloch-Kato conjecture,\\ 
${^ME_{a,b}^1}$ is supported in the area $0 \le a \le d$ and $ b \ge -k$, \\ 
${^\et E_{a,b}^1}$ is supported in the area $0 \le a \le d$ and 
$-c \le b \le a$. \\
So for $k\ge c$ we have isomorphisms everywhere. 
For the special choice $k=c-1$ the spectral sequences only differ by the line $b=-c$ and we have a long exact sequence
(cf. \cite{JS08} or \cite{UzuD}):
\[\begin{array}{cccccccc}
 \cdots \to {^\et E_{a+1,-c}^2}& \to& \H_M^{2d+c-a}(\X,\Z/n(c+d-1)) \\
 &\to& \H_\et^{2d+c-a}(\X,\Z/n(c+d-1)) 
 &\to& {^\et E_{a,-c}^2} \to \cdots.   
\end{array}\]
and we can put $\KH_\et^{2d+c-a}(\X,\Z/n(c+d-1)):={^\et E_{a,-c}^2}$. 
\end{Rem}

\begin{Thm}
Let $F$ be a $N$-local field with final residue characteristic $p=\ch(F_0)$
and complete discrete valuation ring $\O_N$, $S=\Spec(\O_N)$,
 and let $\X$ be a scheme separated and of finite type over $S$ of pure dimension $d=\dim_S(\X)$ and $n \in \Z$ with $p \nmid n$. 
Then we have:
\begin{enumerate}
\item The motivic cohomology groups
\[ \H^r_M(\X,\Z/n(q)) := \CH^q(\X,2q-r;\Z/n),\]
are finite for $0 \le r \le q+1$ or $q \ge N+d-1$. 
\item If $\X$ is regular and proper and flat over $S$, then the étale cycle map
\[ \rho_{\X,n}^{r,q}:\; \H^r_M(\X,\Z/n(q)) \to \H_\et^{r}(\X,\Z/n(q)) \]
is an isomorphism of finite groups in case $0 \le r \le q$ or $q \ge N+d-1$, and is injective for $r=q+1$.
\end{enumerate}
\begin{proof}
1.) By stratifications, excision sequences and $\ell'$-alterations one can assume $X$ to be smooth over $F$. Then by \ref{motivic-rem} one is left to prove the finiteness of the étale cohomology groups and the ${^\et E_{a,-N}^2}$ terms, which by \cite{JSS09} can be identified with the Kato-homology groups $\KH_a^{(N)}$. The latter one was shown in \ref{finite-KH}. The former one follows from \cite{SGA4.5} Th. finitude, Hochschild-Serre spectral sequence and \ref{coh-hdf}.\\
2.) By \ref{motivic-rem} we only have to consider the case $q=N+d-1$, where we have the long exact sequence. But in this case the Kato conjecture {\bf KC} gives the vanishing of $\KH_a^{(N)}(\X)={^\et E_{a,-N}^2}$ for all $a \in \Z$.
\end{proof}
\end{Thm}

\begin{Cor}
Let $F$ be an $N$-local field with final residue characteristic $p=\ch(F_0)$ and $\hat \X$ a projective and flat scheme over the rank-$N$-valuation ring $\O$ of $F$ with poly-quasi-semi-stable reduction and simple normal crossing divisor $\hat Z$.
Then there is a simple normal crossing divisor $\hat Z'$ containing $\hat Z$ such that for the $F$-variety 
$U=(\hat \X \sm \hat Z')_N$ with $d:=\dim U$ the motivic cohomology groups
\[ \H_M^{2d+2+N-a}(U,\Z/n(d+N)) =0\qquad \forall a \le d+1.\]
and the edge morphisms
\[\epsilon: H^{(N)}_a(U,\Z/n) \to \KH^{(N)}_a(U,\Z/n) \]
are isomorphisms for $a \le d+1$ (and trivial surjective for $a \ge d+2$) and all integers $n \in \Z$ with $p \nmid n$. 
Furthermore, it can be achieved that both groups vanish for $a \le d$.
\begin{proof}
For $\epsilon$ to be an isomorphism use Bertini \ref{bertini} and proceed as in the proof of \ref{etale-hasse}. For the vanishing of the groups we are then left to show that the weight homology groups 
$\H_a(\Gamma_{X\sm(W \cup D)})$ vanish for $a < \dim X$ and a smooth connected variety $X$ with a sncd $W \cup D$, where $D=X \cap H$ with the hyperplane section $H$ meeting every irreducible component $V$ of $X$ and $W^{[r]}$, $r \ge 0$, transversally (and connected if $\dim V \ge 2$).
Induction on the number of components of $W$ and dimension of $X$ (see proof of {\bf SAL}) we are left with $X$ an irreducible curve and $D$ a finite number of points. So we have a natural surjection $\pi_0(D) \srj \pi_0(X)$ and the exact sequence:
\[ \H_0(\Gamma_D) \srj \H_0(\Gamma_X) \to \H_0(\Gamma_{X\sm D}) \to \H_{-1}(\Gamma_D)=0, \]
and so $\H_0(\Gamma_{X\sm D})$ vanishes (but $\H_1$ i.g. would not).\\
The vanishing of the motivic cohomology group follows from the long exact sequence in \ref{motivic-rem} and the identification
${^\et E_{a,-c}^2}=\KH_a^{(N)}(U)$ (cf. \cite{JSS09}).
\end{proof}
\end{Cor}

\section{Higher dimensional class field theory}

Let the notations be like in the last chapter: $F$ be the $N$-local field given by $(F_N, \dots, F_0)$, $p=\ch(F_0)$, and their complete discrete valuation rings $\O_m$, $S_m=\Spec(\O_m)$ and $\Cat_m=\Sch_\sft(S_m)$.

\begin{DefLem} 
Let $X$ be scheme separated and of finite type over $F=F_N$. For every closed point $x \in X$ the residue field $\k(x)$ is a $N$-local field. So by \cite{Kat7982} there is a reciprocity map
\[ \rho^x: \K^M_N(\k(x)) \longrightarrow \Gal^\ab_{\k(x)}=\pi_1^\ab(\Spec(\k(x))), \]
where $\K^M_N(\k(x))$ is the $N$-th Milnor $K$-group.
Let $\pi_1^\ab(X)$ be the abelianized étale fundamental group of $X$ (cf. \cite{SGA1} or \cite{Sza09}), which is covariant functorial. Taking the push-forward from $\pi_1^\ab(\Spec(\k(x)))$ to $\pi_1^\ab(X)$ and adding over all closed points, we get a map
\[ \bigoplus_{x \in X_0} \K^M_N(\k(x)) \longrightarrow \pi_1^\ab(X).  \] 
If $X$ is proper over $F$, then in \cite{Sai85} it was shown that this map factors through 
\[\SK_N(X):=\coker\left( \bigoplus_{y \in X_1} \K^M_{N+1}(\k(y)) \stackrel{\partial}{\longrightarrow}  \bigoplus_{x \in X_0} \K^M_N(\k(x)) \right),  \]
where $\partial$ is the boundary map coming from Milnor-$K$-theory and $X_a$ is the set of $a$-dimensional points of $X$.
The induced map 
\[ \rho^X: \SK_N(X) \longrightarrow \pi_1^\ab(X) \]
is called the \emph{reciprocity map} of $X$. 
\end{DefLem}

The main task of higher dimensional class field theory is to understand this map, especially its kernel and cokernel. The method introduced by U. Jannsen and S. Saito (cf. \cite{JS03}) is to interpret the occuring groups as homology groups.\\

\begin{Lem}[Duality] 
\label{duality-hlf}
Let $F$ be a $N$-local field with final residue characteristic $p=\ch(F_0)$.
Let $f: X \to \Spec(F)$ be a separated morphism of finite type and $n$ an integer not divisible by $p$. 
Then for any integer $i$ there is a perfect pairing of finite groups
\[\H^{i}_{\et}(X,D_X(C)) \x \H^{N+1-i}_{\et,c}(X,C) \longrightarrow \H^{N+1}_\et(X,Rf^!\Z/n(N))
\stackrel{\tr_X}{\longrightarrow}  \Z/n,   \]
for any $C$ in $D^b_c(X,\Z/n)$, the derived category of complexes of étale sheaves of $\Z/n$-modules on $X$ with bounded constructible cohomology sheaves, where $D_X(C)=\RHom(C,Rf^!\Z/n(N))$ and $\tr_X$ is induced by the trace map 
\[Rf_*Rf^!\Z/n(N) \to \Z/n(N). \]
\begin{proof}
This follows from Poincaré duality together with the duality of the ground field \cite{Koy99}. Also cp. with \cite{JS03} Lem. 5.3.
\end{proof}
\end{Lem}

\begin{Prp}
\label{facts}
Now assume that $X$ has a 
model $\X$ over $S=S_N$, $Y:=\X_{s,\rd}$.
Let $n$ be an integer not divisible by $p=\ch(F_0)$. Then the $5$-term sequence of the niveau spectral sequence for $\X$, $X$ and $Y$ w.r.t. the étale homology theory $\H=\H_\et^{(N)}$ over $S$ give the following huge (anti-)commutative diagram with exact rows and columns:
\[ \xymatrix{ 
 \KH_3(\X,\Z/n) \ar^-{j^*}[d]\ar[r] & E^2_{1,1}(\X,\Z/n) \ar^-{j^*}[d]\ar[r] & 
 \H_{2}(\X,\Z/n) \ar^-{j^*}[d]\ar[r] & \KH_{2}(\X,\Z/n) \ar^-{j^*}[d], \\
  \KH_3(X,\Z/n) \ar^-{\Delta^2}[d]\ar^-{d^2}[r] & E^2_{1,1}(X,\Z/n) \ar^-{\Delta}[d]\ar^-{\epsilon^X}[r] & \H_{2}(X,\Z/n) \ar^-{\delta}[d]\ar[r] & \KH_2(X,\Z/n) \ar^-{\Delta^1}[d]\ar[r] & 0, \\
\KH_{2}(Y,\Z/n) \ar^-{i_*}[d]\ar^-{d^2}[r] & E^2_{0,1}(Y,\Z/n) \ar^-{i_*}[d]\ar^-{\epsilon^Y}[r] & \H_{1}(Y,\Z/n) \ar^-{i_*}[d]\ar[r] & \KH_{1}(Y,\Z/n) \ar^-{i_*}[d]\ar[r] & 0, \\
\KH_{2}(\X,\Z/n) \ar[r] & E^2_{0,1}(\X,\Z/n) \ar[r] & \H_{1}(\X,\Z/n) \ar[r] & \KH_{1}(\X,\Z/n) \ar[r] & 0. }\]
Now we have the following facts about this diagram:
\begin{enumerate}
\item If $\X$ is proper over $S$ then the duality \ref{duality-hlf} induces the following isomorphisms
 \[\begin{array}{ccccc}
  H_{2}(X,\Z/n) & \iso & \H^1_\et(X,\Z/n)^\vee & = & \pi_1^\ab(X)/n, \\
  H_{1}(Y,\Z/n) & \iso & \H^1_\et(Y,\Z/n)^\vee & = & \pi_1^\ab(Y)/n.
 \end{array} \]
\item The Galois symbol maps $K_i^M(\k)/n \to \H^i(\k,\Z/n(i))$ in degree $i$ are isomorphisms by Bloch-Kato conjecture (cf. \cite{Voe10b}, \cite{Voe10a}, \cite{SJ06}, \cite{HW09}) and commute with the localisation maps (cf. \cite{Kat86} Lem. 1.4). So we get identifications (cf. \cite{JS03} §5):
 \[\begin{array}{ccccc}
  E^2_{1,1}(X,\Z/n) &  = & \SK_N(X)/n, \\
  E^2_{0,1}(Y,\Z/n) &  = & \SK_{N-1}(Y)/n.
 \end{array} \]
\item The map $\epsilon^X$ and $\epsilon^Y$ coincide with the corresponding reciprocity maps modulo $n$ up to sign by \cite{JS03} and \cite{JSS09}.
\item If $\X=\bigcup_i \X_i$ with closed subschemes $\X_i$ such that all $\X^{[r]}$, $r \ge 0$, are regular, proper and flat over $S$ then the Kato conjecture {\bf KC} implies that
 \[ \KH_a(\X,\Z/n)=0 \quad \text{and}\quad \KH_{a+1}(X,\Z/n)\stackrel{\Delta^a}{\iso}\KH_{a}(Y,\Z/n) \quad \forall a \in \Z.\]
\item If $X$ is the generic fibre of a scheme $\hat \X$ which is proper and flat with poly-quasi-semi-stable reduction over the rank-$N$-valuation ring of $F$ (cf. \ref{pqss-red}), or if $X$ is a proper simple normal crossing curve and we inductively take the simple normal crossing models of its reductions \ref{sncmr-curves}, let $\hat \X_0$ be the fibre over the finite field $F_0$. 
Then  by \ref{fibre-eq} we have 
\[ \KH_{a+1}({\hat \X_0},\Z/n)\iso \H_a(\Gamma_{\hat \X_0},\Z/n) \]
and we get the following exact sequence (for $p \nmid n$):
\[ \H_2(\Gamma_{\hat \X_0},\Z/n) \to \SK_N(X)/n \stackrel{\rho^X_n}{\to} \pi_1^{\ab}(X)/n \to \H_1(\Gamma_{\hat \X_0},\Z/n) \to 0. \]
In the case of a curve we, moreover, have $\H_2(\Gamma_{\hat \X_0},\Z/n)=0$.
\end{enumerate}
\end{Prp}

For further analysis we need the following notations and lemmata:

\begin{Not}
\begin{itemize}
\item 
For an abelian group $A$ and a set of primes $\Ll$ let $\N(\Ll)$ be the monoid of all natural numbers which have prime divisiors only in $\Ll$. $\N(\Ll)$ is ordered by divisibility relations.
 We define $A_\Ll$ to be the $\Ll$-completion
\[ A_\Ll :=  \prlim_{n \in \N(\Ll)} A/n. \]
\item We put
\[A_{\Lb\di}:= \bigcap_{n \in \N(\Ll)} n A = \ker \bigl( A \to \prlim_{n \in \N(\Ll)} A/n  =A_\Ll \bigr) \] 
to be the subgroup of all $\Ll$-divisible elements of $A$, and 
$A_{\Lb\Di}$ to be the maximal $\Ll$-divisible subgroup of $A$. We mention that we always have
\[ A_{\Lb\Di} \ins A_{\Lb\di},\]
but in general no equality.
\end{itemize}
\end{Not}

\begin{Not}
\label{nota-mod-div}
Let $F=F_N$ be a $N$-local field with final residue characteristic $p=\ch(F_0)$. Let $\Ll$ be the set of all primes not containing $p$. And let $f: X \to \Spec(F)$ be a connected and proper variety over $F$. We then put:
\begin{itemize}
\item $\K'_a(F):= \K^M_a(F)/\K^M_a(F)_{\Lb\Di}$,
\item $\SK'_N(X):= \SK_N(X)/\SK_N(X)_{\Lb\Di}$,
\item ${'\rho^F}:  \K'_N(F) \to \Gal^\ab_F(p')$ induced by the reciprocity map $\rho^F$,
\item ${'\rho^X}:  \SK'_N(X) \to \pi_1^\ab(X)(p')$ induced by the reciprocity map $\rho^X$,
\item $V_N'(X):=\ker\left(N_{X|F}:\SK'_N(X) \stackrel{f_*}{\to} \K'_N(F)  \right)$ induced by the norm map, 
\item $\pi_1^\geo(X):=\ker \left( \pi_1^\ab(X) \stackrel{f_*}{\to} \Gal_{F}^\ab\right)$, 
\item ${'\rho^X_|}: V_N'(X) \to \pi_1^\geo(X)(p')$ induced by the reciprocity map ${'\rho^X}$.
\end{itemize}
\end{Not}

\begin{Rem}
\label{norm-rec-mod-div}
Note, that the maps in \ref{nota-mod-div} are well-defined, since pro-finite groups have no non-trivial divisible subgroups and images of divisible groups are divisible.\\
We also mention that $\ker'\rho^F=0$, since $\ker \rho^F=\K^M_N(F)_\Di$ by Fesenko, cf. \cite{Fes96} Thm. 2 (ii) and \cite{Ras95} Thm. 2.18. And the cokernel of the norm map $\coker N_{X|F}$ is finite by the reciprocity isomorphisms for higher local fields by Kato, cf. \cite{Ras95} Thm. 2.6. We then have the following commutative diagram
\[\xymatrix{ 
0\ar[r]& \ker'\rho^X_|\ar@{=}[d] \ar[r]& V_N'(X) \ar@{^(->}[d]\ar^-{'\rho^X_|}[r]&  \pi_1^\geo(X)(p')\ar[r]\ar@{^(->}[d] & \coker'\rho^X_| \ar@{^(->}[d] \ar[r]& 0 \\
0\ar[r]& \ker'\rho^X \ar[d]\ar[r]& \SK_N'(X)\ar^-{'\rho^X}[r]\ar^-{N_{X|F}}[d]&  \pi_1^\ab(X)(p') \ar[d]\ar[r] & \coker'\rho^X \ar[d]\ar[r]& 0 \\
 &0 \ar[r]& \K_N'(F)\ar@{^(->}^-{'\rho^F}[r]\ar@{->>}[d]&  \Gal_F^\ab(p')\ar[r] & \coker'\rho^F \ar[r]& 0\\
 &&\coker N_{X|F}.
}\]
\end{Rem}

\begin{Lem}
\label{k-grp-hdf}
Let $F=F_N$ be a $N$-local field with final residue characteristic $p=\ch(F_0)$, $q=\#F_0$. Then for any integer $a$ the Milnor-$K$-groups modulo maximal $p'$-divisible subgroups satisfy
\[ \K'_a(F_N)\iso \K'_a(F_{N-1}) \oplus \K'_{a-1}(F_{N-1}),  \]
leading to
\[ \K'_a(F_N)\iso \Z^{\binom{N}{a}} \oplus \left(\Z/(q-1)\right)^{\binom{N}{a-1}}.  \]
\begin{proof}
By \cite{BT73} Ch. I §4 or \cite{Ras95} §1.2 (4) we have an exact sequence:
 \[0\longrightarrow U^1_a(F_N) \longrightarrow  \K^M_a(F_N) \longrightarrow  \K^M_a(F_{N-1})\oplus \K^M_{a-1}(F_{N-1}) \longrightarrow 0,\]
where $U^1_a(F_N)$ is generated by symbols $\{b_1,\dots,b_a\}$ with $b_1 \in 1+\fm_N$ and the right map depends on a choice of a prime element $\pi$ and is given by sending $\{b_1,\dots,b_a\}+\{b_1',\dots,b_{a-1}',\pi\}$ to $\left(\{\ol b_1,\dots,\ol b_a\},\{\ol b_1',\dots,\ol b_{a-1}'\}\right)$, with $b_i,b_i' \in \O_N^\x$. Since we have
\[ 1+\fm_N \iso  \prlim_{i \ge 1}\left( (1+\fm_N)/(1+\fm^{1+i}_N) \right) \quad\text{ and }\quad (1+\fm_N^i)/(1+\fm^{i+1}_N) \iso F_{N-1}, \] 
we get that $1+\fm_N$ and hence $U^1_a(F_N)$ is $\ell$-divisible for $\ell\neq \ch(F_{N-1})$, thus the desired recursion formula, which leads to 
\[ \K'_a(F_N) \iso \bigoplus_{j=0}^N \K'_{a-j}(F_0)^{\binom{N}{j}}. \]
For finite fields we have $K_0^M(F_0)=\Z$, $K_1^M(F_0)=F_0^\x\iso\Z/(q-1)$ and $\K^M_a(F_0)=0$ for $a\neq 0,1$. Thus the claim.
\end{proof}
\end{Lem}

\begin{Lem}
\label{coh-hdf}
Let $F=F_N$ be a $N$-local field with final residue characteristic $p=\ch(F_0)$, $q=\#F_0$. Put $\Q/\Z' = \bigoplus_{\ell \neq p} \Q_\ell/\Z_\ell$. Fix integers $a,b$.
\begin{enumerate}
\item 
For $N=0$ we have
\[ \H^a(F_0,\Q/\Z'(b)) = \left \{ \begin{array}{llll}
  0 &\text{ for } & a\neq0,1,& \\
 \Q/\Z' &\text{ for } & a=0,& b=0, \\
 \Q/\Z' &\text{ for } & a=1,& b=0, \\
 \Z/(q^b-1) &\text{ for } & a=0,&b\neq0, \\
  0 &\text{ for } & a=1,&b\neq0. \\
\end{array}\right.  \]
\item 
For $N \ge 0$ and $a-b \neq 0,1$ the groups $\H^a(F_N,\Q/\Z'(b))$ are finite of an order dividing $(q^{b-a}-1)^{\binom{N}{a}}$.
\item
For $N \ge 0$, $a-b =1$ we have 
\[ \H^a(F_N,\Q/\Z'(a-1)) \iso (\Z/(q-1))^{\binom{N}{a}} \oplus (\Q/\Z')^{\binom{N}{a-1}}. \]
\item For $N \ge 0$, $a-b=0$ we have
\[ \H^a(F_N,\Q/\Z'(a)) \iso (\Q/\Z')^{\binom{N}{a}}.\]
\end{enumerate}
\begin{proof}
For a finite field $F_0$ with $q$ elements the $b$-twisted cyclotomic character $\chi^b$ of the Frobenius element $\phi_q$ is $\chi^b(\phi_q)=q^b$. Since $G:=\Gal_{F_0} \iso \hat \Z$ and $\cd F_0 =1$ we have 
\[\H^1(F_0,\Q/\Z'(b))=\Q/\Z'(b)_G=(\Q/\Z')/(q^b-1)=0,\]
and 
\[\H^0(F_0,\Q/\Z'(b))=\Q/\Z'(b)^G=\{x \in \Q/\Z'\;|\; (q^b-1)x=0\}=\frac{1}{q^b-1}\Z/\Z,\]
for $b \neq 0$ and $\Q/\Z'$ for $b=0$.\\
For $N \ge 0$ we by \cite{Kat82b} Thm. 3 have the recursion formula: 
\[ \H^a(F_N,\Q/\Z'(a-1)) \iso \H^a(F_{N-1},\Q/\Z'(a-1)) \oplus \H^{a-1}(F_{N-1},\Q/\Z'(a-2))\]
leading to 
\[ \H^a(F_N,\Q/\Z'(a-1)) \iso \bigoplus_{j=0}^N \H^{a-j}(F_0,\Q/\Z'(a-j-1))^{\binom{N}{j}}. \]
The formulas for finite fields now give the claim.\\
For $a-b=0$ the Bloch-Kato conjecture implies the isomorphism
\[\H^a(F_N,\Q/\Z'(a)) \iso \K^M_a(F_N) \otimes \Q/\Z' \iso \K'_a(F_N) \otimes \Q/\Z'.\]
The recursion formula \ref{k-grp-hdf} then gives the claim.\\
For $a-b \neq 0,1$ we use the excision sequence, proper base change and absolute purity on the tripel $\eta=\Spec(F_N)$, $S=\Spec(\O_N)$, $s=\Spec(F_N)$ to get the short exact sequence
\[  \H^a(F_{N-1},\Q/\Z'(b)) \longrightarrow \H^a(F_N,\Q/\Z'(b)) \longrightarrow \H^{a-1}(F_{N-1},\Q/\Z'(b-1)),  \]
which by induction gives the claim on the finiteness of these groups.
\end{proof}
\end{Lem}

\begin{Lem}
\label{pi1-structure}
Let $X$ be a connected scheme proper over a $N$-local field $F$ of final residue characteristic $p=\ch(F_0)$.
Then the prime-to-$p$-part of the abelianized étale fundamental group $\pi_1^\ab(X)(p')$ is finitely generated and has only a finite $p'$-torsion subgroup.
\begin{proof}
Put $\hat\Z'=\prod_{\ell \neq p}\Z_\ell$.
Let $F^\alg$ and $F^\sep$ be an algebraic closure, separable closure, resp., of $F$. Put $X^\sep=X\x_F \Spec(F^\sep)$ and $\ol X = X \x_F\Spec(F^\alg)$.
By \cite{Mil80} III §4 Cor. 4.19 we have a short exact sequence on the prime-to-$p$-part
\[ 0 \longrightarrow (NS(\ol X)_\tors)^*(p') \longrightarrow \pi_1^\ab(\ol X)(p') \longrightarrow T\Alb_{\ol X}(\ol F)(p') \longrightarrow 0,  \]
where $(NS(\ol X)_\tors)^*$ is finite and $T\Alb_{\ol X}(\ol F)(p') \iso {\hat{\Z'}}^{2g}$, $g = \dim \Alb_{\ol X}$.
The dual of the Hochschild-Serre spectral sequence \cite{Mil80} III Thm. 2.20 gives an exact sequence
\[ \H^2(F,\Q/\Z')^\vee \longrightarrow \pi_1^\ab(X^\sep)(p')_{\Gal_F} \longrightarrow \pi_1^\ab(X)(p') \longrightarrow \Gal_F^\ab(p') \longrightarrow 0.\]
Since we have $\pi_1^\ab(X^\sep)=\pi_1^\ab(\ol X)$ for proper $X$, and by \ref{coh-hdf} the group
\[
\Gal^\ab_F(p') \iso \H^1(F,\Q/\Z')^\vee \iso \hat \Z' \oplus (\Z/(q-1))^N,\]
is finitely generated, we get the claim for $\pi_1^\ab(X)(p')$.
\end{proof}
\end{Lem}

The following result is a finiteness result of the first homology groups of the Gersten complex, which are encoded in the groups $E_{a,1}^2$. For its proof an application of Kato's conjecture about the cohomological Hasse principle is needed:

\begin{Prp}
\label{gersten-null}
Let $F=F_N$ be a $N$-local field with complete discrete valuation ring $\O_N$ and final residue characteristic $p=\ch(F_0)$, $q=\#F_0$. 
Let $\X=\bigcup_i \X_i$ with irreducible components $\X_i$ such that all $\X^{[r]}$, $r \ge 0$, are regular, proper and flat over $S=\Spec(\O_N)$ of $\dim_S\X=d+1$ consider the niveau spectral sequence 
$ E^1_{r,q}(\X,\Z/n)
\Longrightarrow \H_{r+q}(\X,\Z/n)$
for the étale homology theory $\H^\et(\_,\Z/n)$ from \ref{etale-hom-def}.
\\
Then there is an integer $C \ge 0$ (independent of $n$) such that 
\[ \#E^2_{a,1}(\X,\Z/n) \;\le\; C  \]
for $a=0,1$ and every integer $n \in \Z$ with $p \nmid n$.
\begin{proof}
First assume that $\X$ is regular and connected and $n=\ell^r$ for a prime $\ell$. If $d=0$ then $\X$ is a finite extension of $S$ and the complex $E^1_{\bullet,1}(\X)$ can like in \cite{Kat86} Lem. 1.4 or \cite{JS03} be identified with the boundary map
\[ \K^M_{N-1}(F_{N-1})/n \longleftarrow \K^M_{N}(F_{N})/n, \]
which by \ref{k-grp-hdf} is surjective with kernel $\K'_{N}(F_{N-1})/n \iso (\Z/(q^b-1))/n$, where $b$ is the dimension of the final residue field of $\X$ over $F_0$. So $E^2_{0,1}(\X)=0$ and $\#E^2_{1,1}(\X)|(q^b-1)$.\\
Now let $d \ge 1$. By {\bf PB} and the $\ell'$-alteration theorem \cite{Ill09} we can assume $\X$ to be regular, projective with quasi-semi-stable reduction. By the Bertini theorem \cite{JS09} there is a horizontal irreducible regular divisor $W$ such that $W \cup \X_{s,\rd}$ is a sncd and $U=\X\sm W$ is affine. By {\bf KC}  we get
\[ 0= \KH_a(\X) \to \KH_a(U) \to \KH_{a-1}(W) =0, \]
so $\KH_a(U)=0$ for $a \le 3$. Thus the spectral sequence induce a commutative diagram 
\[\xymatrix{ 
&\H_{a+1}(\X) \ar[r] \ar^-\epsilon@{->>}[d]& \H_{a+1}(U) \ar^-\epsilon@{->>}[d]\\
E^2_{a,1}(W) \ar[r]& E^2_{a,1}(\X) \ar[r] &E^2_{a,1}(U), }\]
for $a=0,1$, where $\epsilon$ are the edge morphisms, which are surjective for $a=0,1$. By {\bf SAL} we have $\H_{a+1}(U)=0$ for $a+1 \le 2\le d+1$. So $E^2_{a,1}(W)$ surjects onto $E^2_{a,1}(\X)$. So by induction we get $E^2_{0,1}(\X)=E^2_{0,1}(W)=0$ and $\#E^2_{1,1}(\X)|\#E^2_{1,1}(W)|(q^b-1)$.
Since the degree of the used alteration is finite and works for almost all primes $\ell \neq p$, we only need to consider finitely many alterations and the independece of $n$ is clear. \\
Now let $\X=\bigcup \X_i$ then there is a descent spectral sequence like in {\bf DSS}:
\[ E^2_{a-r,1}(\X^{[r]}) \stackrel{r \ge 0}{\Longrightarrow} E^2_{a,1}(\X).  \]
For $0 \le a-r \le 1$ the left hand side is universally bounded, so for $a=0,1$ the right hand side, too. 
\end{proof}
\end{Prp}

\begin{Lem}
\label{finite-norm}
Let $F=F_N$ be a $N$-local field with complete discrete valuation ring $\O_N$ and final residue characteristic $p=\ch(F_0)$, $q=\#F_0$. 
\begin{enumerate}
\item Let $\X=\bigcup_i \X_i$ be a connected scheme with irreducible components $\X_i$ such that all $\X^{[r]}$, $r \ge 0$, are regular, proper and flat
over $S=\Spec(\O_N)$. Let $\Ll$ be a set of prime numbers different from $p$. Then the specialization map 
\[  V'_N(\X_{\eta})_\Ll \to V'_{N-1}(\X_{s,\rd})_\Ll \]
has finite kernel and cokernel.
\item Let $X=\bigcup_iX_i$ be a connected scheme proper over $F$ with irreducible components $X_i$. Then the cokernel of
\[V'_N(X^{[0]}) \to V'_N(X)\]
is finite.
\end{enumerate}
\begin{proof}
1.) Let $X=\X_\eta$ and $Y=\X_{s,\rd}$ and $n\in \N(\Ll)$. As in \ref{facts} we have identifications 
\[\begin{array}{lllll} 
E^2_{1,1}(X,\Z/n) &\iso& \SK_N(X)/n     &\iso&  \SK'_N(X)/n\\
E^2_{0,1}(Y,\Z/n) &\iso& \SK_{N-1}(Y)/n & \iso &\SK'_{N-1}(Y)/n.\end{array}  \]
Consider the diagram where $K_n$ and $C_n$ are the kernels of the right horizontal maps:
\[\xymatrix{   
0\ar[r] & K_n \ar[d]\ar[r]& E^2_{1,1}(\X,\Z/n)\ar[r]\ar[d] & E^2_{1,1}(S,\Z/n) \ar[d] \\
0\ar[r] & V'_N(X)/n \ar[d]\ar[r]& \SK_N'(X)/n \ar[r]\ar[d] & K'_N(F_N)/n \ar[r]\ar[d] &0 \\
0\ar[r] & V'_{N-1}(Y)/n\ar[r] \ar[d]& \SK_{N-1}'(Y)/n\ar[r] \ar[d]& K'_{N-1}(F_{N-1})/n\ar[r]\ar[d] &0 \\
0 \ar[r]& C_n \ar[r]& E^2_{0,1}(\X,\Z/n)\ar[r] & E^2_{0,1}(S,\Z/n). 
} \]
The two middle horizontal sequences are exact up to finite groups bounded by $\#K'_N(F_N)_\tors*\#\coker N_{X|F_N}$, 
$\#K'_{N-1}(F_{N-1})_\tors*\#\coker N_{Y|F_{N-1}}$, resp. (cf. \ref{k-grp-hdf}). By \ref{gersten-null} the $E^2$-terms are finite of a bounded order $C$ independent of $n$. This shows that $K_n$ and $C_n$ are also finite and universally bounded. Now takening the limit over $n \in \N(\Ll)$ give the claim (cf. \cite{For11b} Lem. 3.2).\\
2.) The commutative diagram
\[\xymatrix{   
0\ar[r] & V'_N(X^{[1]}) \ar[d]\ar[r]& \SK_N'(X^{[1]})\ar[r]\ar[d] & K'_N(F)^{(\pi_0(X^{[1]}))}\ar[r]\ar[d] & 0 \\%
0 \ar[r]& V'_N(X^{[0]})\ar[r]\ar^-\phi[d] & \SK_N'(X^{[0]}) \ar[r]\ar@{->>}[d]& K'_N(F)^{(\pi_0(X^{[0]}))} \ar[d]\ar[r]& 0 \\
0\ar[r] & V'_N(X)\ar[r] & \SK_N'(X)\ar[r] & K'_N(F)\ar[r] & 0} \]
has exact rows up to finite groups. The vertical maps are complexes. The middle down map is surjective and the middle right homology group vanishes, since $X$ is connected. So $\coker(\phi)$ is finite by the long exact sequence for the homology groups.
\end{proof}
\end{Lem}

\begin{Lem}
\label{V-tor}
 Let $X$ be a scheme proper and geometrically irreducible over a $N$-local field $F$ of final residue characteristic $p=\ch(F_0)$, then $V_N'(X)$ is a torsion group and has a finite image in $\pi_1^\geo(X)(p')$ under the reciprocity map $'\rho^X$.
\begin{proof}
First let $X$ be a regular curve. Then by \ref{sncmr-curves} one inductively can take simple normal crossing models of $X$ and its reduction. Let $Y$ be the final reduction over the finite field $F_0$. By Bloch, Kato, Saito \cite{KS86} Thm. 6.1 or \cite{Ras95} the group $V'_0(Y)$ and thus $V'_0(Y)_\Ll$ is finite. 
By \ref{finite-norm} then also $V'_N(X)_\Ll$ is finite. Since by definition $V'_N(X)$ has no $\Ll$-divisible subgroup, we have an injection $V'_N(X) \inj V'_N(X)_\Ll$. So $V'_N(X)$ is finite for regular curves $X$.\\
Now one can proceed like \cite{For11b} Prp. 2.1. The claim for curves follows then by taking the nomalization. For any other variety then one can let run a curve through a given cycle after a suitable field extension. The image of the curve is then still torsion.\\
The finite image in $\pi_1^\geo(X)(p')$ then comes from the finiteness of the $p'$-torsion part of $\pi_1^\ab(X)$ by \ref{pi1-structure}.
\end{proof}
\end{Lem}

\begin{Thm}
\label{finite-ker}
Let $F=F_N$ be a $N$-local field with complete discrete valuation ring $\O_N$ with final residue field characteristic $p=\ch(F_0)$.
Let $X$ be a regular connected scheme proper over $F$ of dimension $d \ge 0$.  
Then $V'_N(X)$ and $\ker'\rho^X$ are finite and the kernel $\ker \rho^X$ of the reciprocity map 
\[\rho^X:\quad \SK_N(X) \to \pi_1^\ab(X)\]
is the direct sum of a finite group and a group which is $\ell$-divisible for every $\ell \neq p$.\\
Furthermore, if $d=1$ then $\ker \rho^X$ is divisible by any prime $\ell \neq p$ and the cokernel can be decomposed
\[\coker\rho^X \iso D \oplus \H_1(\G_{Y},\Z_\Ll),\]
where $Y$ is the final fibre of inductively chosen simple normal crossing models (\ref{sncmr-curves}) and $D$ is a uniquely $\ell$-divsible group for any $\ell \neq p$. In particular, $\coker \rho^X$ is $\ell$-torsion free for any $\ell \neq p$.
\begin{proof}
Let $\Ll$ be the set of all primes not containing $p$. 
By \ref{V-tor} and \ref{norm-rec-mod-div} we have that $\ker'\rho^X \ins V'_N(X)$ is torsion of finite index.
By \cite{For11b} Lem. 3.2 we have that also $(\ker'\rho^X)_\Ll$ injects into $V'_N(X)_\Ll$ with finite index. 
And since the torsion subgroups of $\pi_1^\ab(X)(p')$ and $\coker'\rho^X$ are finite by  \ref{pi1-structure} and using higher local class field theory, again by \cite{For11b} Lem. 3.2 we have a exact sequence
\[ 0 \to (\ker'\rho^X)_\Ll \to \SK'_N(X)_\Ll \stackrel{'\rho^X_\Ll}{\longrightarrow} \pi_1^\ab(X)_\Ll \to (\coker'\rho^X)_\Ll\to 0,  \] 
and the limit in \ref{facts} gives rise to a surjection
\[\KH_3(X,\Z_\Ll)= \prlim_{n \in \N(\Ll)}\KH_3(X,\Z/n) \srj (\ker'\rho^X)_\Ll. \]
By \ref{alteration} there is an alteration
 $\pi: X' \to X$ such that $X'$ is regular and the generic fibre of a projective scheme $\hat \X$ over the rank-$N$-valuation ring of $F$ with poly-quasi-semi-stable reduction. If we remove the finite number of primes dividing the degree $r:=\deg(\pi)$ from $\Ll$ we by {\bf PB} get a commutative diagram of surjections
(with $Y:=\hat\X_0$): 
\[\xymatrix{
 \H_2(\Gamma_{Y},\Z_\Ll) \ar^-\sim@{=}[r] & \KH_3(X',\Z_\Ll) \ar@{->>}[r] \ar@{->>}[d]& \KH_3(X,\Z_\Ll) \ar@{->>}[d] \\
&(\ker'\rho^{X'})_\Ll \ar[r]& (\ker'\rho^X)_\Ll.
}\]
So the bottom map is a surjection. By \ref{finite-norm} we see that the maps
\[ V'_N(X)_\Ll \injl (\ker'\rho^X)_\Ll \srjl (\ker'\rho^{X'})_\Ll \inj V'_N(X')_\Ll \inj\dots \inj V'_{0}(Y)_\Ll \srjl V'_{0}({Y}^{[0]})_\Ll  \]
are injective, surjective, resp., up to finite groups. And $V'_{0}({Y}^{[0]})_\Ll$ is finite by Bloch, Kato, Saito (cf. \cite{KS86} Thm. 6.1 or \cite{Ras95}).\\
The finite number of primes $\ell \mid r$ can be dealt with in the same way, but separately, by taking $\ell'$-alterations \ref{alteration}. 
So it follows that $V'_N(X)_\Ll$ and $(\ker'\rho^X)_\Ll$ are finite groups for the set $\Ll$ of all primes excluding $p$. Like in \ref{V-tor} $V'_N(X)$ injects into $V'_N(X)_\Ll$ and therefore is finite as well. \\
Now consider the following diagram of exact sequences:
\[\xymatrix{
  0 \ar[r]&C \ar@{^(->}[r]\ar@{^(->}[d]& \SK_N(X)_{\Lb\di} \ar@{^(->}[d]\ar^{\rho^X_p}[r]\ar[r]& \pi_1^\ab(X)(p)\ar@{^(->}[d] \ar[r]& D \ar@{^(->}[d] \\
 0\ar[r]&\ker \rho^X \ar@{^(->}[r]\ar[d]& \SK_N(X) \ar[d]\ar^-{\rho^X}[r]& \pi_1^\ab(X)  \ar@{->>}[d]\ar@{->>}[r]& \coker\rho^X \ar[r]\ar[d]&0 \\
0 \ar[r] &(\ker'\rho^X)_\Ll \ar@{^(->}[r]& \SK_N(X)_\Ll \ar^-{\rho^X_\Ll}[r]\ar@{->>}[d] & \pi_1^\ab(X)_\Ll \ar[d]\ar@{->>}[r]& (\coker'\rho^X)_\Ll \ar[r] \ar[d]& 0\\
 &&D' \ar[r]& 0 \ar[r] &0,\\ 
}\]
where $C:=\ker \rho^X \cap \SK_N(X)_{\Lb\di}$. For any $\ell \neq p$ the $\ell$-torsion-part of $\pi_1^\ab(X)_\Ll$ is finite by \ref{pi1-structure}. From \cite{JS03} Lem. 7.7 then follows that $\SK_N(X)_{\Lb\di}$ is $\Ll$-divisible. Since $\pi_1^\ab(X)(p)$ is $\Ll$-torsion free, also $C$ is $\Ll$-divisible. So $\ker \rho^X$ is the direct sum of $C$ and a subgroup of $(\ker'\rho^X)_\Ll$, which is finite.\\
Now let $X$ be a curve ($d=1$). Then $\KH_3(X,\Z_\Ll)$ and thus $(\ker'\rho^X)_\Ll$ vanish, so $\ker\rho^X=C$ in the diagram above, which is $\Ll$-divisible.
Furthermore, we have
\[(\coker'\rho^X)_\Ll \iso \KH_2(X,\Z_\Ll) \iso \H_1(\Gamma_Y,\Z_\Ll).  \]
By the snake-lemma we get an exact sequence
\[ \pi_1^\ab(X)(p) \to D \to D' \to 0.   \]
By \cite{For11b} Lem. 3.4 (2) the group $D'$ is divisible. It follows that $D$ is a $\Ll$-divisible group. Since 
$\H_1(\G,\Z_\Ll)=\H_1(\G,\Z)\otimes \Z_\Ll$ is a free $\Z_\Ll$-module for curves, we get
\[\coker\rho^X \iso D \oplus \H_1(\Gamma_Y,\Z_\Ll).\]
\end{proof}
\end{Thm}

\begin{Cor}
Let $F$ be an $N$-local field with final residue characteristic $p=\ch(F_0)$ and $X=\bigcup_i X_i$ a connected proper variety over $F$ with regular irreducible components $X_i$. Then $V'_N(X)$ is finite.
\begin{proof}
This follows from \ref{finite-ker} and \ref{finite-norm}.
\end{proof}
\end{Cor}



\providecommand{\bysame}{\leavevmode\hbox to3em{\hrulefill}\thinspace}
\providecommand{\MR}{\relax\ifhmode\unskip\space\fi MR }
\providecommand{\MRhref}[2]{%
  \href{http://www.ams.org/mathscinet-getitem?mr=#1}{#2}
}
\providecommand{\href}[2]{#2}

\end{document}